\documentclass{amsart}
\usepackage{amsmath,amssymb,amsthm,amsfonts,amscd,mathrsfs,mathtools}
\usepackage[utf8]{inputenc}
\usepackage[dvipsnames]{xcolor}
\usepackage{tikz-cd}
\usetikzlibrary{shapes.arrows,trees,backgrounds,decorations.markings,shadows,calc, patterns, positioning}
\usepackage[hyphens]{url}
\usepackage{caption}
\usepackage{hyperref}
\hypersetup{
linkbordercolor=blue,
pdfborderstyle={/S/U/W 1}
}
\usepackage{color}

\theoremstyle{plain}
    \newtheorem{thm}{Theorem}[section]
    \newtheorem{cor}[thm]   {Corollary}
    \newtheorem{prop}[thm]  {Proposition}
    \newtheorem{remark}[thm] {Remark}
    
    \newtheorem{lem}[thm]   {Lemma}

    \newtheorem{atheorem}{Theorem}
    
    \newtheorem{acorollary}{Corollary}

\theoremstyle{definition}
    \newtheorem{defn}[thm]  {Definition}

    \newtheorem{nota}[thm]{Notation}

\newcommand{\cA}{\mathcal{A}} 
\newcommand{\fb}{\mathfrak{b}} 
\newcommand{\C}{\mathbb{C}} 
\newcommand{\mD}{\mathring{D}} 
\newcommand{\F}{\mathbb{F}} 
\newcommand{\cF}{\mathcal{F}} 
\newcommand{\bH}{\mathbb{H}} 
\newcommand{\cH}{\mathcal{H}} 
\newcommand{\bP}{\mathbb{P}} 
\newcommand{\Q}{\mathbb{Q}} 
\newcommand{\R}{\mathbb{R}} 
\newcommand{\cS}{\mathcal{S}} 
\newcommand{\fS}{\mathfrak{S}} 
\newcommand{\fw}{\mathfrak{w}} 
\newcommand{\Z}{\mathbb{Z}} 

\newcommand{\Br}{\mathfrak{Br}} 

\newcommand{\Hur}{\mathrm{Hur}}
\newcommand{\Hn}{\mathrm{Hn}}

\usepackage{stmaryrd} 
\newcommand{\qq}{\sslash} 
\usepackage{mathbbol}
\DeclareSymbolFontAlphabet{\mathbb}{AMSb}
\DeclareSymbolFontAlphabet{\mathbbl}{bbold}
\newcommand{\one}{\mathbbl{1}} 

\newcommand{\ua}{\underline{a}} 
\newcommand{\ualpha}{\underline{\alpha}} 

\newcommand{\pa}[1]{\left(#1\right)}
\newcommand{\set}[1]{\left\{#1\right\}}
\DeclarePairedDelimiter\floor{\lfloor}{\rfloor}

\renewcommand{\phi}{\varphi}
\renewcommand{\epsilon}{\varepsilon}

\DeclareMathOperator{\Aut}{Aut} 
\DeclareMathOperator{\Hom}{Hom} 
\newcommand{\Tor}{\mathrm{Tor}}
\DeclareMathOperator{\lst}{lst} 
\DeclareMathOperator{\rst}{rst} 
\newcommand{\Spec}{\mathrm{Spec}}

\begin{document}

\title[Polynomial stability of Hurwitz spaces]{Polynomial stability of the homology of Hurwitz spaces}

\author{Andrea Bianchi}
\email{anbi@math.ku.dk}
\address{Department of Mathematical Sciences, University of Copenhagen \newline
Universitetsparken 5, Copenhagen, 2100, Denmark}  

 \author{Jeremy Miller}
  \email{jeremykmiller@purdue.edu}
\address{Purdue University, Department of Mathematics\newline
150 North University, West Lafayette IN, 47907, USA}


\date{\today}

\begin{abstract}
For a finite group $G$ and a conjugation-invariant subset $Q\subseteq G$, we consider the Hurwitz space $\mathrm{Hur}_n(Q)$ parametrising branched covers of the plane with $n$ branch points, monodromies in $G$ and local monodromies in $Q$. For $i\ge0$ we prove that $\bigoplus_n H_i(\mathrm{Hur}_n(Q))$ is a finitely generated module over the ring $\bigoplus_n H_0(\mathrm{Hur}_n(Q))$. As a consequence, we obtain polynomial stability of homology of Hurwitz spaces: taking homology coefficients in a field, the dimension of $H_i(\mathrm{Hur}_n(Q))$ agrees for $n$ large enough with a quasi-polynomial in $n$, whose degree is easily bounded in terms of $G$ and $Q$. Under suitable hypotheses on $G$ and $Q$, we prove classical homological stability for certain sequences of components of Hurwitz spaces.
Our results generalise previous work of Ellenberg-Venkatesh-Westerland, and rely on techniques introduced by them and by Hatcher-Wahl.
\end{abstract}

\maketitle
 
\section{Introduction}
Let $G$ be a finite group and let $Q\subseteq G$ be a conjugation-invariant subset.
For $n\ge0$, the Hurwitz space $\Hur_n(Q)$ considered in this article is
a certain homotopy quotient of the set $Q^n$ of $n$-tuples of elements in $Q$ by an action of the braid group $\Br_n$: see Definition \ref{defn:Hurwitz}. The homotopy type of $\Hur_n(Q)$ coincides with that of the moduli space of certain ``decorated'' branched covers of the complex plane $\C$; see Subsection \ref{subsec:connectionalggeo} for more details, and for the link between the Hurwitz spaces of this article and the Hurwitz spaces usually considered in algebraic geometry.

We are broadly interested in stability properties of homology of components of Hurwitz spaces. Suitable stability results for $H_i(\Hur_n(Q))$ can find applications in enumerative number theory: the main instance of this is are the work of Ellenberg-Venkatesh-Westerland on the Cohen-Lenstra heuristics \cite{EVW:homstabhur} and the work of Ellenberg-Tran-Westerland on the Malle conjecture \cite{ETW:shufflealgebras}.
The results of this article are an attempt to generalise the topological part of \cite{EVW:homstabhur}, as we do not require $Q\subseteq G$ to be a single conjugacy class with the ``non-splitting property'', and we consider homology in any Noetherian ring $R$; yet we have not been able to find a counterpart to the nice linear stability ranges from \cite{EVW:homstabhur}, making our results of a more qualitative than quantitative nature; in Subsection \ref{subsec:connectionalggeo} we briefly explain why this prevents applications of our results in enumerative number theory.

\subsection{Statement of results}
The disjoint union $\Hur(Q)=\coprod_{n\ge0}\Hur_n(Q)$ has a natural structure of topological monoid,\footnote{The multiplication is strict in the model of Hurwitz spaces that we use; in other natural models one would only have an $E_1$-algebra.} which is recalled in Subsection \ref{subsec:topmonoid}.
As a consequence, $H_*(\Hur(Q))=\bigoplus_{i\ge0,n\ge0}H_i(\Hur_n(Q))$ admits a natural structure of bigraded ring, where the two gradings are the homological degree $i$ and the ``weight'' $n$. In particular the direct sum
\[
 A=H_0(\Hur(Q))=\bigoplus_{n\ge0}H_0(\Hur_n(Q))
\]
admits a structure of graded ring, and for all $i\ge0$ and commutative ring $R$, the direct sum $H_i(\Hur(Q);R)=\bigoplus_{n\ge0}H_i(\Hur_n(Q);R)$ admits a structure of $A\otimes R$-bimodule, in particular of left $A\otimes R$-module. Our first main result is the following.
\begin{atheorem}
 \label{thm:main1}
 Let $R$ be a Noetherian commutative ring. For $i\ge0$, the module $H_i(\Hur(Q);R)$ is finitely generated over $A\otimes R=H_0(\Hur(Q);R)$.
\end{atheorem}
A quantitative consequence of the previous theorem is the following
corollary: see Definition \ref{defn:kGQ} for the constant $k(G,Q)\ge1$. We will also make use of the notion of quasi-polynomials, see Subsection \ref{subsec:quasipolynomials}.

\begin{acorollary}
\label{cor:main1}
 Let $\F$ be a field and let $i\ge0$. Let $\ell\ge1$ be such that $a^\ell=\one\in G$ for all $a\in Q$. Then there is a quasi-polynomial $p^\F_i(t)$
 of degree at most $k(G,Q)-1$ and period dividing $\ell$ such that, for $n$ large enough,
 $\dim_{\F}H_i(\Hur_n(Q);\F)=(p^\F_i)_{[n]_\ell}(n)$.
\end{acorollary}
We note that the condition $k(G,Q)=1$ is the ``non-splitting property'' for  the couple $(G,Q)$, introduced in \cite{EVW:homstabhur}: if $k(G,Q)=1$ then $Q$ is a single conjugacy class in $G$ and for every subgroup $H\subseteq G$, the intersection $Q\cap H$ is either empty or a single conjugacy class in $H$. In this case, $\dim_{\F}H_i(\Hur_n(Q);\F)$ is eventually equal to a quasi-polynomial in $n$ of degree $0$, i.e. it eventually coincides with a periodic function of $n$: this is an abstract version of homology stability, in the sense that for $n,n'$ large enough with respect to $i$, if $n-n'$ is divisible by $\ell$ then the vector spaces $H_i(\Hur_n(Q);\F)$ and $H_i(\Hur_{n'}(Q);\F)$ are abstractly isomorphic, but no explicit isomorphism is provided. We remark that in the case $k(G,Q)=1$, Theorem \ref{thm:main1} is proved in \cite{EVW:homstabhur}, under the additional hypotheses that homology is taken with coefficients in $\Q$ (or at least $\Z[\frac 1{|G|}]$): more precisely, for a suitable constant $D\ge1$ an explicit map $H_i(\Hur_n(Q);\Q)\to H_i(\Hur_{n+D}(Q);\Q)$ is constructed, and for suitable constants $A,B\ge1$ this map is shown to be an isomorphism provided that $n\ge Ai+B$, i.e. a linear stability range is proved.
The given stable isomorphism $H_i(\Hur_n(Q);\Q)\to H_i(\Hur_{n+D}(Q);\Q)$ is defined as a \emph{sum} of maps induced in homology by topological stabilisation maps $\Hur_n(Q)\to \Hur_{n+D}$.

The philosophy that stability results should be formulated in terms of finiteness results as modules over algebras was popularized by the field of representation stability (see e.g. \cite{CEF:FImodules, CEFN:FImodules}). As is common in representation stability, we leverage the fact that the algebra governing stability $A\otimes R$ is Noetherian. This helps us analyze a certain spectral sequence constructed using an arc complex introduced by Hatcher-Wahl \cite{hatcherwahl}. These spectral sequences also appeared in work of Ellenberg-Venkatesh-Westerland \cite{EVW:homstabhur}: our viewpoint on stability allows us to analyze this spectral sequence in new cases, in particular when the non-splitting property of $Q\subseteq G$ fails. We achieve this by considering at the same time several possible stabilisation maps (algebraically, this corresponds to working over polynomial rings in several variables), an idea appearing already in \cite{ADCK:edgestab,KMT:extremal}.

For $\omega \in G$ let us denote by $\Hur(Q)_\omega\subset\Hur(Q)$ the subspace characterised by the total monodromy being equal to $\omega$ (see Definition \ref{defn:totmon}).
The ring $A=H_0(\Hur(Q))$ contains a subring $A_\one=H_0(\Hur(Q)_\one)$, and
for $\omega \in G$, multiplication in $H_*(\Hur(Q))$ makes the submodule $H_i(\Hur(Q)_\omega)\subset H_i(\Hur(Q))$ into a left $A_\one$-module, for all $i\ge0$. Our second main result is the following.
\begin{atheorem}
\label{thm:main2}
Let $R$ be a Noetherian commutative ring. Let $i\ge0$ and let $\omega\in G$; then $H_i(\Hur(Q)_\omega;R)$ is finitely generated over $A_\one\otimes R=H_0(\Hur(Q)_\one;R)$.
\end{atheorem}
A quantitative consequence of the previous theorem is the following
(see Definition \ref{defn:kGQomega} for the constant $k(G,Q,\omega)\ge1$).
\begin{acorollary}
 \label{cor:main2}
 Let $\F$ be a field and let $i\ge0$. Let $\ell\ge1$ be such that $a^\ell=\one\in G$ for all $a\in Q$. Then there is a quasi-polynomial $p^\F_{i,\omega}$
 of degree at most $k(G,Q,\omega)-1$ and period dividing $\ell$ such that, for $n$ large enough,
 $\dim_{\F}H_i(\Hur_n(Q)_\omega;\F)=(p^\F_{i,\omega})_{[n]_\ell}(n)$.
\end{acorollary}
We compare Theorems \ref{thm:main1} and \ref{thm:main2} through the following example: let $G=\fS_d$ be the symmetric group on $d$ elements, and let $Q\subset G$ denote the conjugacy class of transpositions. Then $k(G,Q)=\floor{d/2}$, and for $d\ge4$ Theorem \ref{thm:main1} only predicts quasi-polynomial growth in $n$ of the Betti numbers $\dim_{\F}H_i(\Hur_n(Q);\F)$; in particular these Betti numbers can attain arbitrarily high values for $n\to\infty$.
However, if $\omega=(1,\dots,d)$ is the standard \emph{long cycle}, then $k(G,Q,\omega)=1$ and Theorem \ref{thm:main2} predicts that the Betti number $\dim_{\F}H_i(\Hur_n(Q)_\omega;\F)$ eventually coincides with a periodic function of $n$.

Our third, main result is a classical homological stability result for certain sequences of components of Hurwitz spaces.
\begin{atheorem}
 \label{thm:main3}
 Let $R$ be a commutative ring. Assume that $Q\subset G$ is a single conjugacy class and that there is an element $\omega\in G$ which is \emph{large} with respect to $Q$ (see Definition \ref{defn:omegalarge}). Let $i\ge0$, and let $\ell\ge1$ be such that $a^\ell=\one\in G$ for all $a\in Q$.
 For $a\in Q$ denote by $\lst(a)$ the map $\Hur(Q)\to\Hur(Q)$ induced by left multiplication by a point in $\Hur_1(Q)_a$; see also Definition \ref{defn:stabmap}.\footnote{In our model for Hurwitz spaces, $\Hur_1(Q)_a$ consists of precisely one point.}
Then for $n$ sufficiently large compared to $i$ the stabilisation map
\[
 \lst(a)^\ell_*\colon H_i(\Hur_n(Q)_\omega;R)\to H_i(\Hur_{n+\ell}(Q)_\omega;R)
\]
is independent of $a\in Q$ and is an isomorphism.
\end{atheorem}
Theorem \ref{thm:main3} applies for instance to the following settings:
\begin{itemize}
 \item $G=\fS_p$ for a prime number $p$, $Q$ is the conjugacy class of transpositions, and $\omega=(1,\dots,p)$ is the long cycle;
 \item $G=\Z/d\rtimes\Z/2$ for $d$ odd, $Q$ is the conjugacy class of involutions, and $\omega$ is a generator of $\Z/d$.
\end{itemize}
We remark that Tietz \cite[Theorem 2]{Tietz} also obtains a homology stability result (with an explicit linear stable range) for the integral homology of certain components of Hurwitz spaces, also generalising \cite{EVW:homstabhur}; even though our notion of ``large element of a group with respect to a conjugacy class'' seems comparable with his notion of ``collection of conjugacy classes that invariably generate a group'', we consider our result rather disjoint from his result.

Theorem \ref{thm:main3} can be applied together with the group-completion theorem in order to compute some stable homology groups, giving a partial, positive answer to \cite[Conjecture 1.5]{EVW:homstabhur}.
\begin{nota}
 For $\alpha\in\pi_0(\Hur(Q))$, we denote by $\Hur_\alpha(Q)\subseteq\Hur(Q)$ the connected component $\alpha$.
\end{nota}
\begin{acorollary}
\label{cor:main3}
Let $R$ be a commutative ring.
Let $G,Q,\omega$ be as in Theorem \ref{thm:main3}, and let $a\in Q$.
Fix $\alpha\in\pi_0(\Hur(Q)_\omega)\subseteq\pi_0(\Hur(Q))$.
Then for $n$ sufficiently large compared with $i$ the map induced by the group completion induces a homology isomorphism
\[
 H_i(\Hur_{\hat a^{\ell n}\alpha};R)\cong H_i(\Omega_0 B\Hur(Q);R),
\]
where $\Omega_0 B\Hur(Q)$ denotes the zero component of the group completion of $\Hur(Q)$.
\end{acorollary}
The rational homology of a component $\Omega_0 B\Hur(Q)$ of $\Omega B\Hur(Q)$ is $\Q$ in degrees 0 and 1, and 0 in all other degrees. This statement is first implicitly claimed in \cite[Conjecture 1.5]{EVW:homstabhur}, and a strategy of proof is contained in \cite[Subsection 5.6]{EVW:homstabhurII}, where it is shown that the statement follows from \cite[Theorem 2.8.1]{EVW:homstabhurII}. However, we believe that the proof of \cite[Theorem 2.8.1]{EVW:homstabhurII} contains a gap, which probably can be fixed. We refer to \cite[Corollary 5.4]{ORW:Hurwitz} and \cite[§6.5]{Bianchi:Hur3} for alternative proofs of the statement.

\begin{remark}
 After a first version of this article was circulated, Davis-Schlank \cite[Proposition 3.36]{DavisSchlank} proved that for a field $\F$ and for a finitely generated, weighted $A\otimes\F$-module $M$, the dimensions $\dim_\F(M_n)$ agree, for $n$ large, with a \emph{polynomial} in $n$ of degree bounded as in Corollary \ref{cor:main1}. Our Theorem \ref{thm:main1} allows to apply this result to $M=H_i(\Hur(Q);\F)$ and thus improve the statement of Corollary \ref{cor:main1} by replacing the word ``quasi-polynomial'' with the word ``polynomial''.
\end{remark}

\subsection{Connections to algebraic geometry and enumerative number theory}
\label{subsec:connectionalggeo}
The space $\Hur_n(Q)$ considered in this article is homotopy equivalent to the certain moduli space of regular branched covers of the complex plane $p\colon \cF\to\C$ endowed with the following data:
\begin{enumerate}
 \item an identification of the deck transformation group $\Aut(\cF,p)$ with $G$;
 \item a trivialisation over a suitable lower half-plane $\bH\subset\C$ of the restricted $G$-principal bundle $p^{-1}(\bH)\cong G\times \bH$, up to replacing $\bH$ by smaller and smaller half-planes,
\end{enumerate}
such that there are precisely $n$ branch points, and such that local monodromies around the branch points have values in $Q$. A branched cover is ``regular'' if its group of deck transformations acts transitively on each fibre; we use the term ``$G$-cover'' for a regular branched cover endowed with a decoration as (1) above.

Romagny-Wewers \cite{RomagnyWewers} describe how to construct, for $n\ge0$ and a finite group $G$, a scheme $\cH_{n,G}$ of finite type over $\Spec(\Z)$ whose associated analytic space $\cH_{n,G}(\C)$ can be identified with the moduli space of branched $G$-covers as above, but with $\bP^1_\C$ as target,
without the decoration (2), and without the restriction that local monodromies be in $Q$. The construction already appears in Wewers' PhD thesis \cite{Wewers}, and it is preceeded by work of Fulton \cite{Fulton} and of Fried-V\"olklein \cite{FriedVoelklein}: Fulton constructs for $d\ge1$ and $n\ge0$ a scheme $\cH_{d,n}$ of finite type over $\Spec(\Z)$ whose complex points isomorphism classes of degree-$d$ simple branched covers of $\bP^1_\C$ with $n$ branch points, and Fried-V\"olklein construct the basechange $\cH_{n,G}\times_{\Spec(\Z)}\Spec(\Q)$, which is of finite type over $\Spec(\Q)$.

Ellenberg-Venkatesh-Westerland \cite{EVW:homstabhur} show that at least when $Q\subset G$ is a \emph{rational} conjugation-invariant subset (that is, for all $a\in Q$ and all $n\ge1$ coprime with the order of $a$ in $G$, one has $a^n\in Q$), then Wewers' Hurwitz scheme $\cH_{n,G}$ contains a subscheme $\cH_{n,G}^Q$ whose complex points are isomorphism classes of branched $G$-covers of $\bP^1_\C$ with local monodromies lying in $Q$. Similarly, assuming again that $Q$ is rational, they show the existence of a subscheme $\Hn_{G.n}^Q$ of $\cH_{n+1,G}$
whose complex points are isomorphism classes of branched $G$-covers of $\C$ with local monodromies in $Q$.
The only missing ``decoration'' is (2) in the list above, namely a trivialisation of the $G$-cover over some lower half-plane; it follows that $\Hn^Q_{G,n}(\C)$ should be thought of as corresponding to the quotient $\Hur_n(Q)/G$, where $G$ acts by global conjugation on the set $Q^n$, the action is compatible with that of $\Br_n$, and hence we obtain an induced action of $G$ on $\Hur_n(Q)$. It would be interesting to know, without any restriction on $Q\subseteq G$, whether there exists a scheme $\tilde{\Hn}^Q_{G,n}$, possibly of finite type over $\Spec(\Z)$ (and ideally, an affine scheme) whose associated analytic space $\tilde{\Hn}^Q_{G,n}(\C)$ is homotopy equivalent to $\Hur_n(Q)$.

In \cite{EVW:homstabhur}, rational homological stability for $\Hur_n(Q)$ is proved under the non-splitting hypothesis, and an \emph{explicit linear stability range} is provided. This leads to an exponential upper bound on the size of the rational cohomology of $\Hur_n(Q)$ that does not depend on $n$: there is a constant $C>0$ such that in each degree $i\ge0$ one has $\dim_\Q H^i(\Hur_n(Q);\Q)<C^i$. This in turn leads to a similar upper bound on the size of $H^i(\Hur_n(Q)/G;\Q)\cong H^i(\Hn_{G,n}^Q(\C);\Q)$, when $Q$ is a rational conjugacy class of $G$. For a finite field $\F_q$ with algebraic closure $\bar\F_q$, the latter upper bound can be used to control the size of the \'etale cohomology of the basechange $\Hn_{G,n}^Q\times_{\Spec(\Z)}\Spec(\bar\F_q)$. One can then use the Grothendieck-Lefschetz trace formula to express the $|\Hn^Q_{G,n}(\F_q)|$ as the alternating sum of the traces of the $\F_q$-Frobenius acting on the \'etale cohomology of $\Hn_{G,n}^Q\otimes_{\Spec(\Z)}\Spec(\bar\F_q)$. The above bounds on the dimension of the \'etale cohomology groups, together with the Deligne bounds on the size of the eigenvalues of the Frobenius, give an estimate of $|\Hn^Q_{G,n}(\F_q)|$, and describe, at least for $q$ large, how $|\Hn^Q_{G,n}(\F_q)|$ grows for $n\to\infty$.

The lack of an explicit linear stability range in the results of this article seems to exclude the possibility to employ our results in a similar framework as that of \cite{EVW:homstabhur}.
However, if an explicit linear stable range could be established, these kinds of polynomial stability results plausibly would have implications for point count problems.

\subsection{Acknowledgments}
The first author would like to thank B\'eranger Seguin for a useful conversation on the algebraic-geometric theory of Hurwitz spaces, and Oscar Randal-Williams for a conversation about stability phenomena in the presence of several stabilisation maps. Both authors would like to thank Jordan Ellenberg and Craig Westerland for useful comments on a first draft of the article.

Andrea Bianchi was supported by the Danish National Research Foundation through the Centre for Geometry and Topology (DNRF151) and the European Research Council under the European Union Horizon 2020 research and innovation programme (grant agreement No. 772960).

Jeremy Miller was supported in part by NSF grant DMS-2202943 and a Simons Foundation Collaboration Grants for Mathematicians.

\section{Preliminaries}
\label{sec:preliminaries}
Throughout the article, $G$ denotes a finite group and $Q\subseteq G$ a conjugation-invariant subset.
We denote by $q_1,\dots,q_m$ the elements of $Q$, where $m=|Q|$, and we let $\ell\ge1$ be a positive integer such that $a^\ell=\one\in G$ for each $a\in Q$: for instance, we could take $\ell$ to be the least common multiple of the multiplicative orders in the group $G$ of the elements of $Q$.
\subsection{Hurwitz spaces as homotopy quotients}
For $n\ge0$, we denote by $\Br_n$ the Artin braid group on $n$ strands,
with generators $\sigma_1,\dots,\sigma_{n-1}$ satisfying the usual braid and commuting relations.
The group $\Br_n$ acts on the set $Q^n$ of $n$-tuples of elements of $Q$ as follows: for $1\le i\le n-1$, the standard generator $\sigma_i\in\Br_n$ sends
\[
 \sigma_i\colon (a_1,\dots,a_n)\mapsto (a_1,\dots,a_{i-1},a_{i+1},a_i^{a_{i+1}},a_{i+2},\dots,a_n),
\]
where, for $a,b\in Q$, we denote $a^b=b^{-1}ab\in Q$. The fact that the action is well-defined is a consequence of the fact that $Q$, with the operation of conjugation restricted from $G$, is a quandle.
The following is the definition of Hurwitz spaces that we are going to use throughout the article.
\begin{defn}
 \label{defn:Hurwitz}
For $n\ge0$, we define the Hurwitz space $\Hur_n(Q)$ as the homotopy quotient
\[
 \Hur_n(Q)=Q^n\qq\Br_n.
\]
\end{defn}
A priori, a homotopy quotient is only defined as a homotopy type; in this article we realise homotopy quotients by the standard bar construction; for example $\Hur_n(Q)$, as a concrete topological space, is the geometric realisation of the simplicial set $B_\bullet(*,\Br_n,Q^n)$.

\subsection{Components of Hurwitz spaces}
By definition, $\Hur_n(Q)$ is the homotopy quotient of the set $Q^n$ by the action of the discrete group $\Br_n$: it follows that $\pi_0(\Hur_n(Q))$ is in natural bijection with the set of orbits of the action of $\Br_n$ on the set $Q^n$.
\begin{defn}
 For $\ua=(a_1,\dots,a_n)\in Q^n$ we denote by $\Hur_n(Q,\ua)\subset\Hur_n(Q)$ the component corresponding to the orbit of $\ua$ under the action of $\Br_n$.
\end{defn}
The space $\Hur_n(Q,\ua)$ is aspherical.
Let us denote by $\Br_n\cdot\ua\subset Q^n$ the orbit of $\ua$ under the braid group action: then $\Hur_n(Q,\ua)$ is canonically homeomorphic to $(\Br_n\cdot \ua)\qq\Br_n=|B_\bullet(*,\Br_n,\Br_n\cdot\ua)|$, and the fundamental group of $(\Br_n\cdot \ua)\qq\Br_n$ based at the 0-simplex $\ua$ is canonically isomorphic to the subgroup of $\Br_n$ stabilising $\ua$, which we denote $\Br_n(\ua)\subseteq\Br_n$.

\subsection{Topological monoid structure}
\label{subsec:topmonoid}
For $n,m\ge0$, consider the standard concatenating map of sets
$Q^n\times Q^m\overset{\cong}{\to} Q^{n+m}$
and the standard concatenating map of braid groups $\Br_n\times\Br_m\hookrightarrow\Br_{n+m}$.
If we let $\Br_n\times\Br_m$ act on $Q^{n+m}$ through its inclusion into $\Br_{n+m}$, we have that the map $Q^n\times Q^m\overset{\cong}{\to} Q^{n+m}$ is $(\Br_n\times\Br_m)$-equivariant; we also say that the map of sets $Q^n\times Q^m\overset{\cong}{\to} Q^{n+m}$ is equivariant with respect to the map of groups $\Br_n\times\Br_m\hookrightarrow\Br_{n+m}$. This equivariance gives rise to a map between the homotopy quotients
\[
 \Hur_n(Q)\times\Hur_m(Q)\cong (Q^n\times Q^m)\qq(\Br_n\times\Br_m)\to Q^{n+m}\qq\Br_{n+m}=\Hur_{n+m}(Q),
\]
and these maps assemble into a topological monoid structure on the disjoint union
\[
\Hur(Q)=\coprod_{n\ge0}\Hur_n(Q),
\]
with unit given by the unique point of $\Hur_0(Q)$.

Taking connected components, we obtain a discrete monoid $\pi_0(\Hur(Q))$; if we denote by $\hat q_i$ the connected component of $\Hur(Q)$ given by the (contractible) space $\Hur_1(Q,q_i)$, for all $q_i\in Q$, we have the following presentation by generators and relations of $\pi_0(\Hur(Q))$ as an associative, unital monoid:
\[
 \pi_0(\Hur(Q))\cong\left<\hat q_1,\dots, \hat q_m\  |\  \hat q_i\cdot\hat q_j=\hat q_j\cdot \widehat{q_i^{q_j}}\right>.
\]
\begin{defn}\label{defn:pi0ell}
 We denote by $\pi_0^\ell(\Hur(Q))\subseteq\pi_0(\Hur(Q))$ the unital submonoid generated by the elements $\hat q_j^\ell$ for $1\le i\le m$.
\end{defn}
\begin{lem}\label{lem:relationsB}
The monoid $\pi_0^\ell(\Hur(Q))$ is commutative, and for all $1\le i,j\le m$ the following relation holds in $\pi_0^\ell(\Hur(Q))$:
\[
 \hat q_i^\ell\cdot\hat q_j^\ell=\widehat{q_i^{q_j}}^\ell\cdot\hat q_j^\ell
\]
\end{lem}
\begin{proof}
 We note that the elements $\hat q_i^\ell$ are central elements of the monoid $\pi_0(\Hur(Q))$: indeed for all $j$ we have $\hat q_j\cdot \hat q_i^\ell=\hat q_i^\ell\cdot \widehat{q_j^{q_i^\ell}}=\hat q_i^\ell\cdot\hat q_j$. This implies in particular that $\pi_0^\ell(\Hur(Q))$ is a commutative monoid.
 
 Moreover for all $i,j$ we have in $\pi_0(\Hur(Q))$ the equality $\hat q_i^\ell\cdot\hat q_j=\hat q_j\cdot \widehat{q_i^{q_j}}^\ell$, which by the previous argument is also equal to $\widehat{q_i^{q_j}}^\ell\cdot \hat q_j$; multiplying both sides on right by $\hat q_j^{\ell-1}$ we obtain the described relations among the generators of $\pi_0^\ell(\Hur(Q))$.
\end{proof}

\subsection{Constants attached to finite groups}
The following constants, attached to a group $G$, a conjugation-invariant subset $Q\subset G$, and an element $\omega\in G$, will be used to bound the degree of the quasi-polynomials that govern the growth of the homology of Hurwitz spaces.
\begin{defn}
\label{defn:kGQ}
 We denote by $k(G,Q)\ge1$ the maximum, for $H\subseteq G$ ranging among all subgroups of $G$, of the number of conjugacy classes in $H$ in which the conjugation-invariant subset $Q\cap H\subseteq H$ decomposes.
\end{defn}
\begin{defn}
\label{defn:kGQomega}
We denote by $k(G,Q,\omega)\ge1$ the maximum, for $H\subseteq G$ ranging among all subgroups of $G$ containing $\omega$, of the number of conjugacy classes in $H$ in which the conjugation-invariant subset $Q\cap H\subseteq H$ decomposes.
\end{defn}

\begin{defn}
\label{defn:omegalarge}
 Let $\omega\in G$ and assume that $Q\subset G$ is a single conjugacy class. We say that $\omega$ is \emph{large} with respect to $Q$ if for every $a\in Q$ the elements $\omega$ and $a$ generate $G$.
\end{defn}
Note that if $\omega\in G$ is large with respect to a conjugacy class $Q\subset G$, then in particular $k(G,Q,\omega)=1$. As an example, let $p$ be a prime number, let $G=\fS_p$ and let $Q$ be the conjugacy class of transpositions: then the long cycle $(1,\dots,p)$ is large with respect to $Q$. Another example is the following: let $d$ be an odd number and let $G=\Z/d\rtimes\Z/2$ be the $d$\textsuperscript{th} dihedral group; let $Q$ be the conjugacy class of involutions; then any generator of $\Z/d\subset G$ is large with respect to $Q$.

\subsection{Invariants of components}\label{subsec:invcomponents}
The action of $\Br_n$ on $Q^n$ preserves the following invariants defined on the set $Q^n$:
\begin{itemize}
 \item the \emph{total monodromy}: this invariant associates with an $n$-tuple $(a_1,\dots,a_n)$ the product $\omega:=a_1\dots a_n\in G$;
 \item the \emph{image subgroup}: this invariant associates with $(a_1,\dots,a_n)$ the subgroup $H:=\left<a_1,\dots,a_n\right>\subseteq G$;
 \item the \emph{conjugacy class partition}: let $(a_1,\dots,a_n)\in Q^n$ and let $H=\left<a_1,\dots,a_n\right>$ as above; let $Q_1,\dots,Q_s\subset Q\cap H$ be the conjugacy classes in $H$ in which the conjugation invariant subset $Q\cap H$ splits; this invariant associates with $(a_1,\dots,a_n)$ the splitting $n=n_1+\dots+n_s$, where $n_i$ is the cardinality of $\set{a_1,\dots,a_n}\cap Q_i$. This is called the ``multidiscriminant'' in \cite{EVW:homstabhurII}.
\end{itemize}
Each of the above invariant gives rise to an invariant of connected components of $\Hur(Q)$; we will introduce notation only for the first invariant.
\begin{defn}
\label{defn:totmon}
 For $\omega\in G$ and for $n\ge0$ we denote by $\Hur_n(Q)_\omega\subset\Hur_n(Q)$ the union of connected components corresponding to $n$-tuples $(a_1,\dots,a_n)$ with total monodromy $\omega$.
 Similarly, we denote $\Hur(Q)_\omega=\coprod_{n\ge0}\Hur_n(Q)_\omega$.
\end{defn}

\subsection{Quasi-polynomials}\label{subsec:quasipolynomials}
 Let $\ell\ge1$, and for $n\in\Z$ denote by $[n]_\ell\in\Z/\ell$ the class of $n$ modulo $\ell$. We will use the following notion of quasi-polynomial.
\begin{defn}
A quasi-polynomial of period dividing $\ell$, denoted $p_{[-]_\ell}(t)$, is the datum of $\ell$ polynomials $p_{[1]_\ell}(t),\dots,p_{[\ell]_\ell}(t)\in\Q[t]$.

The degree of a non-zero quasi-polynomial is the maximum among the degrees of the polynomials $p_{[i]_\ell}(t)$; the degree of the zero quasi-polynomial is set to be $-\infty$.

A quasi-polynomial induces a function $\Z\to\Q$, given by sending $n\mapsto p_{[n]_\ell}(n)$.
\end{defn}
As one can see, the argument $n\in\Z$ must be input twice, once as index of the polynomial, once as value of the variable $t$, in order to evaluate a quasi-polynomial at $n$.

\section{Noetherian rings}
In this short section we prove that $A=H_0(\Hur(Q))$ is a Noetherian ring, and give a proof of Theorem \ref{thm:main2} assuming Theorem \ref{thm:main1}.

\subsection{Several subrings of \texorpdfstring{$A$}{A}}
The ring $A$ admits the following presentation by generators and relation as an associative ring (compare with the presentation of $\pi_0(\Hur(Q))$ from Subsection \ref{subsec:topmonoid}):
 \[
  A=\Z\left< [q_1],\dots, [q_m]\  |\  [q_i][q_j]=[q_j][q_i^{q_j}]\right>,
 \]
where $[q_i]\in H_0(\Hur_1(Q))$ is defined as the ground class of the (contractible) space $\Hur_1(Q,q_i)$.

We introduce two subrings of $A$.
\begin{defn}\label{defn:B}
We denote by $B\subseteq A$ the subring generated by the elements $[q_i]^\ell$, for $1\le i\le m$.
Equivalently, $B$ is the monoid ring of the monoid $\pi_0^\ell(\Hur(Q))$ from Definition \ref{defn:pi0ell}.
\end{defn}
\begin{defn}\label{defn:Aone}
Recall Definition \ref{defn:totmon}. We denote by $A_\one\subseteq A$ the monoid ring of the submonoid
 \[
\pi_0(\Hur(Q)_\one)\subseteq\pi_0(\Hur(Q)).
\]
\end{defn}
We observe that $B\subseteq A_\one$, as each generator $\hat q_i^\ell\in \pi_0^\ell(\Hur(Q))$ has total monodromy equal to $q_i^\ell=\one\in G$. We also observe that $A_\one$ is a central subring of $A$, as a consequence of the fact that $\pi_0(\Hur(Q)_\one)$ is a central submonoid of $\pi_0(\Hur(Q))$: to see this, let $a\in Q$ and let $(a_1,\dots,a_n)\in Q^n$ be such that $a_1\dots a_n=\one\in G$; then $\hat a\cdot (\hat a_1\dots\hat a_n)=(\hat a_1\dots\hat a_n)\cdot\widehat {a^{a_1\dots a_n}}$ by the relations holding in $\pi_0(\Hur(Q))$, and we have 
$\widehat {a^{a_1\dots a_n}}=\widehat{a^\one}=\hat a$.
\begin{lem}
 \label{lem:Anoetherian}
The associative ring $A$ is finitely generated as a $B$-module. As a consequence we have the following:
\begin{enumerate}
 \item $A$ is Noetherian: a sub-module of a finitely generated left or right $A$-module is finitely generated.
 \item $A_\one$ is also finitely generated as a $B$-module, and is also a Noetherian ring.
 \item $A$ is finitely generated as a $A_\one$-module.
\end{enumerate}
\end{lem}
\begin{proof}
The commutative ring $B$ is Noetherian, as it is a quotient of a polynomial ring over $\Z$ with $m$ variables (in the same way as $B$ is a quotient of a free abelian monoid on $m$ generators). Thus if we prove that $A$ is a finitely generated $B$-module, we immediately have that $A$ is Noetherian: for if $M$ is a finitely generated (left or right) $A$-module and $M'\subset M$ is a submodule, then $M$ is also finitely generated over $B$, and by Noetherianity of $B$ we have that $M'$ is finitely generated over $B$, and a fortiori over $A$. This proves that the main statement implies (1); it also implies (2), as $A_\one$ is a sub-$B$-module of $A$, and $B$ is Noetherian; Noetherianity of $A_\one$ then follows from the same argument used to prove (1). Finally, (3) follows from the inclusion $B\subseteq A_\one$ together with the fact that $A$ is finitely generated as a $B$-module.

We now prove the main statement. Recall that $\pi_0^\ell(\Hur(Q))$ is a central submonoid in $\pi_0(\Hur(Q))$: this implies that $B$ is a central subring of $A$.
Our next goal is to show that $A$ is a finitely generated $B$-module: more precisely, the products $[a_1]\cdot \dots\cdot[a_k]$ with $k\le m(\ell-1)$ and $a_1,\dots,a_k\in Q$ suffice to generate $A$ over $B$. For this, let $[a_1]\cdot \dots\cdot[a_n]$ be any product of generators of $A$, and assume $n>m(\ell-1)$. By the pigeonhole principle there is $q_j\in Q$ and there are $\ell$ indices $i_1,\dots,i_\ell$ such that $a_{i_1}=\dots=a_{i_\ell}=q_j$. We can then use the relations in $A$ and rewrite
$[a_1]\cdot \dots\cdot[a_n]$ in the form $[q_j]^\ell\cdot [a'_1]\cdot \dots\cdot[a'_{n-\ell}]$, where the sequence $a'_1,\dots,a'_{n-\ell}$ is obtained by removing from $a_1,\dots,a_n$ the elements $a_{i_1},\dots,a_{i_\ell}$, and by suitably conjugating the remaining elements by powers of $q_j$. This concludes the proof that $A$ is a finitely generated $B$-module.
\end{proof}
If $R$ is a commutative ring, we can tensor $A,A_\one,B$ with $R$ and obtain the rings $A\otimes R=H_0(\Hur(Q);R)$, $A_\one\otimes R=H_0(\Hur(Q)_\one;R)$ and $B=R[\pi_0^\ell(\Hur(Q))]$. Lemma \ref{lem:Anoetherian} gives the following corollary.
\begin{cor}\label{cor:Anoetherian}
 Let $R$ be a commutative ring; then $A\otimes R$ and $A_\one\otimes R$ are finitely generated $B\otimes R$-modules, and $A\otimes R$ is a finitely generated $A_\one\otimes R$-module.
 If $R$ is Noetherian, then all rings $A\otimes R$, $A_\one\otimes R$ and $B\otimes R$ are Noetherian.
\end{cor}

\subsection{Proof of Theorem \ref{thm:main2} assuming Theorem \ref{thm:main1}}
Let $R$ be a Noetherian ring and let $i\ge0$; then by Theorem \ref{thm:main1} $H_i(\Hur(Q);R)$ is finitely generated over $A\otimes R$; by Lemma \ref{lem:Anoetherian} $A\otimes R$ is finitely generated as an $A_\one\otimes R$-module, and hence also $H_i(\Hur(Q);R)$ is finitely generated as an $A_\one\otimes R$-module. We have a finite direct sum decomposition of $A_\one\otimes R$-modules
\[
 H_i(\Hur(Q);R)\cong \bigoplus_{g\in G}  H_i(\Hur(Q)_g;R),
\]
and hence each direct summand is a finitely generated $A_\one\otimes R$-module.

\section{Arc complexes}
\label{sec:arccomplex}
Hatcher and Wahl \cite{hatcherwahl} introduced, for $n\ge0$, an augmented semisimplicial set $\cA_{n,\bullet}$, whose initial application was an alternative proof of homology stability of braid groups \cite[Proposition 1.7]{hatcherwahl}, originally proved by Arnol'd \cite{arnold} (see also \cite[§5.6]{RWW}). In the following definition we recall this construction and generalise it to our Hurwitz setting; this is approach has been already used in \cite{EVW:homstabhur}
\subsection{High connectivity of augmented arc complexes}
We denote by $D=[0,1]^2\subset\R^2$ the standard unit square in the complex plane, endowed with basepoint $*=(0,0)$, and denote by $\mD=(0,1)^2$ the interior of $D$. We also let $I=[0,1]\times \set{0}$ be the bottom edge of $D$.

For $n\ge0$ and $1\le i\le n$ we let $\bar z_{n,i}=(i/(n+1),1/2)\in\mD$, and we denote $\bar P_n=\set{\bar z_{n,i}}$.
\begin{defn}
\label{defn:arccomplex}
Let $n\ge0$. For $p\ge-1$ we denote by $\cA_{n,p}$ the set of isotopy classes of collections of $p+1$ arcs $\alpha_0,\dots,\alpha_p\colon[0,1]\to D$ satisfying the following conditions:\footnote{Two collections of arcs are considered isotopic if they are connected by an isotopy through collections of arcs with the required properties.}
\begin{itemize}
 \item each arc $\alpha_i$ is an embedding of $[0,1]$ in $D$, sending $0$ to a point of $I$, $(0,1)$ inside $\mD\setminus\bar P_n$, and $1$ to a point of $\bar P_n$;
 \item the arcs have disjoint images, also at their endpoints, and using the natural orientation of $I$ we have $\alpha_0(0)<\dots<\alpha_p(0)$.
\end{itemize}
In particular, $\cA_{n,-1}$ is a singleton (the empty collection of arcs), and $\cA_{n,p}$ is empty for $p\ge n$. Forgetting arcs makes the collection $\cA_{n,\bullet}$ into an augmented semisimplicial set.

We denote by $\cA_n(Q)_\bullet$ the augmented semisimplicial set $\cA_{n,\bullet}\times Q^n$, whose set of $p$-simplices is $\cA_{n,p}\times Q^n$.

The braid group $\Br_n$ acts both on the augmented semisimplicial set $\cA_{n,\bullet}$ and on the set $Q^n$, hence it acts diagonally on $\cA_n(Q)_\bullet$ by automorphisms of augmented semisimplicial sets; taking levelwise the homotopy quotient we obtain an augmented semisimplicial space $\cA_n(Q)_\bullet\qq\Br_n$.
\end{defn}
The braid group $\Br_n=\Gamma_{0,1}^{(n)}$ acts both on the augmented semisimplicial set $\cA_{n,p}$ and on the set $Q^n$, and hence there is a diagonal action of $\Br_n$ on $\cA_n(Q)_\bullet$.
The geometric realisation of $\cA_{n,\bullet}$ is contractible,
i.e. the augmentation $|\cA_{n,\bullet\ge0}|\to\cA_{n,-1}$ is a
homotopy equivalence, the second space being a point \cite[Theorem 2.48]{Damiolini} (see also \cite[Proposition 3.2]{HV:tethers}).

It also follows that the augmentation $|\cA_n(Q)_{\bullet\ge0}|\to \cA_n(Q)_{-1}$ is a homotopy equivalence, the second space being the set $Q^n$,
and it further follows that the augmentation
$|\cA_n(Q)_{\bullet\ge0}\qq\Br_n|\to\cA_n(Q)_{-1}\qq\Br_n$ is a
homotopy equivalence, the second space being the space $\Hur_n(Q)$.
In the next subsection we will analyse more closely the augmented semisimplicial space $\cA_n(Q)_\bullet\qq\Br_n$.
\begin{nota}
\label{nota:cSndot}
 We denote by $\cS^n_\bullet$ the augmented semisimplicial space $\cA_n(Q)_\bullet\qq\Br_n$, leaving $Q$ implicit.
\end{nota}
We record the previous discussion as a lemma for future reference.
\begin{lem}\label{lem:damiolini}
The augmentation $|\cS^n_\bullet|\to\cS^n{-1}$ is a homotopy equivalence.
\end{lem}

\subsection{The augmented semisimplicial space \texorpdfstring{$\cS^n_\bullet$}{Sndot}}
For $-1\le p\le n-1$, the space of $p$-simplices in $\cS^n_\bullet$ is $(\cA_{n,p}\times Q^n)\qq\Br_n$, in particular it is the homotopy quotient of a set by an action of a discrete group: it has therefore the homotopy type of a disjoint union of aspherical spaces, one for each orbit of the diagonal action of $\Br_n$ on $\cA_{n,p}\times Q^n$. The action of $\Br_n$ on $\cA_{n,p}$ is transitive. Let $\bar\ualpha_{n,p}=(\bar\alpha_{n,p,0},\dots,\bar\alpha_{n,p,p})$ be the isotopy class of the collection of $p+1$ straight vertical segments joining $I$ with the points $\bar{z}_{n,1},\dots,\bar{z}_{n,p+1}\in\bar P_n$ (see Figure \ref{fig:stdsimplex}); then the stabiliser of $\bar\ualpha_{n,p}$
is the subgroup of $\Br_n$ generated by the last $n-p-2$ standard generators, which is isomorphic to $\Br_{n-p-1}$.
\begin{figure}[ht]
 \begin{tikzpicture}[scale=4,decoration={markings,mark=at position 0.3 with {\arrow{>}}}]
  \fill[black, opacity=.2] (0,0) rectangle (1,1);
  \draw[black] (0,0) rectangle (1,1);
  \node at (0,0) {$*$};
  \node at (.2,.5){$\bullet$};
  \node at (.4,.5){$\bullet$};
  \node at (.6,.5){$\bullet$};
  \node at (.8,.5){$\bullet$};
  \node at (.2,.55){$\bar{z}_{4,1}$};
  \node at (.4,.55){$\bar{z}_{4,2}$};
  \node at (.6,.55){$\bar{z}_{4,3}$};
  \node at (.8,.55){$\bar{z}_{4,4}$};
  \draw[thin, looseness=1] (.1,0) to[out=90, in=-90] node{$\bar\alpha_{4,2,0}$} (.2,.5); 
  \draw[thin, looseness=1] (.5,0) to[out=90, in=-90] node{$\bar\alpha_{4,2,1}$} (.4,.5); 
  \draw[thin, looseness=1] (.9,0) to[out=90, in=-90] node{$\bar\alpha_{4,2,2}$} (.6,.5); 
\end{tikzpicture}
 \caption{The 2-simplex $\bar\ualpha_{4,2}$.}
 \label{fig:stdsimplex}
\end{figure}
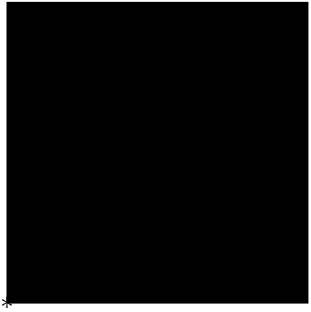

\begin{nota}
We denote by $\one_n\subset\Br_n$ the trivial subgroup of the $n$\textsuperscript{th} braid group. For $0\le p\le n$ we consider $\Br_p\times\Br_{n-p}$ as a subgroup of $\Br_n$, generated by all standard generators except the $p$\textsuperscript{th}.
In particular we denote by $\one_p\times\Br_{n-p}\subset\Br_n$ the subgroup of $\Br_n$ generated by the last $n-p-1$ standard generators.
\end{nota}
\begin{lem}
\label{lem:htypeAnQpqqBrn}
 The spaces $\cS^n_p$ and $Q^{p+1}\times\Hur_{n-p-1}(Q)$ are homotopy equivalent. More precisely, the natural map
$Q^{p+1}\times\Hur_{n-p-1}(Q)\to \cS^n_p=\cA_n(Q)_p\qq\Br_n$ induced by the inclusion of sets
$\set{\bar\ualpha_{n,p}}\times Q^n\subset \cA_{n,p}\times Q^n$ and the inclusion of groups $\Br_{n-p-1}\cong\one_{p+1}\times\Br_{n-p-1}\subset\Br_n$, is a homotopy equivalence.
\end{lem}
\begin{proof}
By the previous discussion, each orbit of the action of $\Br_n$ on $\cA_{n,p}\times Q^n$ contains elements of the form $(\bar\ualpha_{n,p};\ua)$. Moreover,
any two elements $(\bar\ualpha_{n,p};\ua)$ and $(\bar\ualpha_{n,p};\ua')$ in the same $\Br_n$-orbit can be transformed into each other by the action of a suitable element in $\one_{p+1}\times\Br_{n-p-1}$: this implies the equality $a_i=a'_i$ for $0\le i\le p$, and it also implies that the subsequences $(a_{p+1},\dots,a_n)$ and $(a'_{p+1},\dots,a'_n)$ belong to the same orbit of the action of $\Br_{n-p-1}$ on $Q^{n-p-1}$. Viceversa, two elements $(\bar\ualpha_{n,p};\ua)$ and $(\bar\ualpha_{n,p};\ua')$ satisfying the previous requirements can be transformed into each other by the action of $\Br_{n-p-1}\subset\Br_n$, and thus belong to the same orbit of the action of $\Br_n$ on $\cA_n(Q)_p$.

We conclude by remarking that the stabiliser of a single element $(\bar\ualpha_{n,p};\ua)\in\cA_n(Q)_p$ is the subgroup $\one_{p+1}\times \Br_{n-p-1}(a_{p+1},\dots,a_n)\subseteq\one_{p+1}\times\Br_{n-p-1}\subseteq\Br_n$.
\end{proof}
In particular, as already remarked, the space of $(-1)$-simplices $\cS^n_{-1}$ is homotopy equivalent, and in fact canonically homeomorphic, to $\Hur_n(Q)$.

Let us now look at the face maps $d_i\colon \cS^n_p\to \cS^n_{p-1}$, for $p\ge0$ and $0\le i\le p$ (in the case $p=0$ we denote by $d_0$ the augmentation).
\begin{defn}
\label{defn:stabmap}
Let $a\in Q$ and $n\ge0$, and consider $\Br_n\cong \one_1\times\Br_n$ as a subgroup of $\Br_{n+1}$. The map $a\times -\colon Q^n\to Q^{n+1}$, sending $(a_1,\dots,a_n)\mapsto(a,a_1,\dots,a_n)$, is equivariant with respect to the action of $\Br_n$ on $Q^n$ and, by restriction, on $Q^{n+1}$. It thus induces a map on homotopy quotients
\[
 \lst(a)\colon\Hur_n(Q)\to\Hur_{n+1}(Q).
\]
We define similarly a map $\rst(a)\colon \Hur_n(Q)\to\Hur_{n+1}(Q)$ by considering the map of sets $-\times a\colon Q^n\to Q^{n+1}$, sending $(a_1,\dots,a_n)\mapsto(a_1,\dots,a_n,a)$, which is equivariant with respect to the inclusion of groups $\Br_n\times\one_1\subset\Br_{n+1}$.
\end{defn}
\begin{defn}
Let $0\le p\le n-1$ and $0\le i\le p$.
We define a map of spaces
 \[
 \lst_i\colon Q^{p+1}\times\Hur_{n-p-1}(Q)\to Q^p\times\Hur_{n-p}(Q)
 \]
as the map induced on homotopy quotients by the map of sets $Q^n\overset{\cong}{\to} Q^n$ given by
 \[
  (a_0,\dots,a_{n-1})\mapsto (a_0,\dots,a_{i-1},a_{i+1},\dots,a_p,a_i^{a_{i+1}\dots a_p},a_{p+1},\dots,a_{n-1}),
 \]
 which is equivariant with respect to the inclusion of groups $\one_{p+1}\times \Br_{n-p-1}\subset\one_p\times\Br_{n-p}$.
\end{defn}
\begin{prop}\label{prop:descriptiondi}
The following diagram is commutative up to homotopy, where the vertical maps are the homotopy equivalences given by Lemma \ref{lem:htypeAnQpqqBrn}.
 \[
  \begin{tikzcd}
   Q^{p+1}\times\Hur_{n-p-1}(Q)\ar[r,"{\lst_i}"]\ar[d,"\simeq"] &
   Q^p\times\Hur_{n-p}(Q)\ar[d,"\simeq"]\\
   \cS^n_p\ar[r,"{d_i}"] & \cS^n_{p-1}.
  \end{tikzcd}
 \]
\end{prop}
\begin{proof}
Let $\bar\ualpha_{n,p}^{\hat i}\in\cA_{n,p-1}$ be the collection of arcs $(\bar\alpha_{n,p,0},\dots,\hat{\bar\alpha}_{n,p,i},\dots,\bar\alpha_{n,p,p})$ obtained from $\bar\ualpha_{n,p}$ by forgetting $\bar\alpha_{n,p,i}$ (see Figure \ref{fig:facesimplex}, left): then in $\cA_{n,\bullet}$ we have $d_i(\bar\ualpha_{n,p})=\bar\ualpha_{n,p}^{\hat i}$.
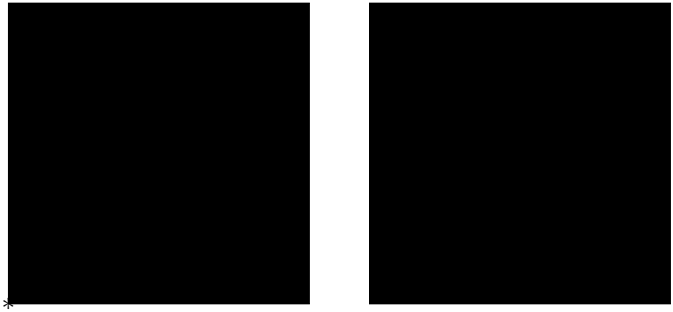
\begin{figure}[ht]
 \begin{tikzpicture}[scale=4,decoration={markings,mark=at position 0.3 with {\arrow{>}}}]
  \fill[black, opacity=.2] (0,0) rectangle (1,1);
  \draw[black] (0,0) rectangle (1,1);
  \node at (0,0) {$*$};
  \node at (.2,.5){$\bullet$};
  \node at (.4,.5){$\bullet$};
  \node at (.6,.5){$\bullet$};
  \node at (.8,.5){$\bullet$};
  \node at (.2,.55){$\bar{z}_{4,1}$};
  \node at (.4,.55){$\bar{z}_{4,2}$};
  \node at (.6,.55){$\bar{z}_{4,3}$};
  \node at (.8,.55){$\bar{z}_{4,4}$};
  \draw[thin, looseness=1] (.5,0) to[out=90, in=-90] node{$\bar\alpha_{4,2,1}$} (.4,.5); 
  \draw[thin, looseness=1] (.9,0) to[out=90, in=-90] node{$\bar\alpha_{4,2,2}$} (.6,.5);
 \begin{scope}[shift={(1.2,0)}]
  \fill[black, opacity=.2] (0,0) rectangle (1,1);
  \draw[black] (0,0) rectangle (1,1);
  \node at (.2,.5){$\bullet$};
  \node at (.4,.5){$\bullet$};
  \node at (.6,.5){$\bullet$};
  \node at (.8,.5){$\bullet$};
  \draw[looseness=1,->] (.2,.55) to[out=50,in=130] (.6,.55);
  \draw[looseness=1,->] (.39,.45) to[out=-130,in=-50] (.2,.45);
  \draw[looseness=1,->] (.6,.45) to[out=-130,in=-50] (.41,.45);
\end{scope}
\end{tikzpicture}
 \caption{On left, the 1-simplex $\bar\ualpha_{4,2}^{\hat 0}$; on right, the product of standard generators $\fb_{4,2,0}=\sigma_2\sigma_1$.}
 \label{fig:facesimplex}
\end{figure}

Consider moreover the product of standard generators
\[
\fb_{n,p,i}=\sigma_p\sigma_{p-1}\sigma_{p-2}\dots\sigma_{i+1}\in\Br_{p+1}\times\one_{n-p-1}\subset\Br_n,
\]
depicted in Figure \ref{fig:facesimplex}, right: this is the empty product, i.e. the neutral element in $\Br_n$, for $i=p$. Then the action of $\fb_{n,p,i}\in\Br_n$ on $\cA_{n,p-1}$ sends $\bar\ualpha_{n,p}^{\hat i}\mapsto\bar\ualpha_{n,p-1}$. It follows that the stabiliser of $\bar\ualpha_{n,p}^{\hat i}$ in $\Br_n$ is
$\pa{\one_p\times\Br_{n-p}}^{\fb_{n,p,i}}$. Repeating the argument of Lemma \ref{lem:htypeAnQpqqBrn} with $\bar\ualpha_{n,p}^{\hat i}$ instead of $\bar\ualpha_{n,p-1}$, we obtain the following: the inclusion of sets
$\set{\bar\ualpha_{n,p}^{\hat i}}\times Q^n\subset \cA_{n,p-1}\times Q^n$ and the inclusion of groups $\pa{\one_p\times\Br_{n-p}}^{\fb_{n,p,i}}\subset\Br_n$ give rise to a homotopy equivalence
\[
 \pa{\set{\bar\ualpha_{n,p}^{\hat i}}\times Q^n}\qq \pa{\one_p\times\Br_{n-p}}^{\fb_{n,p,i}} \overset{\simeq}{\to} \cA_n(Q)_{p-1}\qq\Br_n=\cS^n_{p-1},
\]
and the following square commutes on the nose, where the top map is induced on homotopy quotients by the obvious bijection $\set{\bar\ualpha_{n,p}}\times Q^n\cong \set{\bar\ualpha_{n,p}^{\hat i}}\times Q^n$, which is equivariant with respect to the injective group homomorphism
$\one_{p+1}\times \Br_{n-p-1}\cong\pa{\one_{p+1}\times \Br_{n-p-1}}^{\fb_{n,p,i}}\subset \pa{\one_p\times\Br_{n-p}}^{\fb_{n,p,i}}$:
 \[
  \begin{tikzcd}
   Q^{p+1}\times\Hur_{n-p-1}(Q)\ar[r]\ar[d,"\simeq"] &
   \pa{\set{\bar\ualpha_{n,p}^{\hat i}}\times Q^n}\qq \pa{\one_p\times\Br_{n-p}}^{\fb_{n,p,i}} \ar[d,"\simeq"]\\
   \cS^n_p=\cA_n(Q)_p\qq\Br_n\ar[r,"{d_i}"] & \cS^n_{p-1}=\cA_n(Q)_{p-1}\qq\Br_n.
  \end{tikzcd}
 \]
We can then construct a strictly commutative square, whose vertical maps are homotopy equivalences and whose horizontal maps are homeomorphisms
 \[
  \begin{tikzcd}[column sep=35pt]
   \pa{\set{\bar\ualpha_{n,p}^{\hat i}}\times Q^n}\qq \pa{\one_p\times\Br_{n-p}}^{\fb_{n,p,i}} \ar[r,"\fb_{n,p,i}\cdot-","\cong"']\ar[d,"\simeq"] &
   \pa{\set{\bar\ualpha_{n,p-1}}\times Q^n}\qq \pa{\one_p\times\Br_{n-p}} \ar[d,"\simeq"]\\
   \cS^n_{p-1}=\cA_n(Q)_{p-1}\qq\Br_n\ar[r,"\fb_{n,p,i}\cdot-","\cong"'] & \cS^n_{p-1}=\cA_n(Q)_{p-1}\qq\Br_n.
  \end{tikzcd}
 \]
The top horizontal map is induced by the bijection of sets $\fb_{n,p,i}\colon \pa{\set{\bar\ualpha_{n,p}^{\hat i}}\times Q^n} \cong \pa{\set{\bar\ualpha_{n,p-1}}\times Q^n}$, given by the action of $\fb_{n,p,i}$, together with the isomorphism of groups
$\pa{\one_p\times\Br_{n-p}}^{\fb_{n,p,i}}\cong \pa{\one_p\times\Br_{n-p}}$, given by conjugation by $\fb_{n,p,i}^{-1}$ inside $\Br_n$; note that the aforementioned bijection of sets sends
\[
(\bar\ualpha_{n,p}^{\hat i}; a_0,\dots, a_{n-1})\mapsto (\bar\ualpha_{n,p-1};a_0,\dots,a_{i-1},a_{i+1},\dots,a_p,a_i^{a_{i+1}\dots a_p},a_{p+1},\dots,a_{n-1}).
\]

The bottom horizontal map has a similar description, but involving the automorphism of the set $\cA_n(Q)_{p-1}$ given by the action of $\fb_{n,p,i}$, together with the inner automorphism of the group $\Br_n$ give by conjugation by $\fb_{n,p,i}^{-1}$.

We now appeal to the following standard fact: if a group $G$ acts on a set $X$ and if $g\in G$, then the homeomorphism $g\cdot-\colon X\qq G\to X\qq G$, induced by the action of $g$ on $X$ and by conjugation by $g^{-1}$ on $G$, is homotopic to the identity of $X\qq G$. In our case, the bottom horizontal map $\fb_{n,p,i}\cdot-$ in the last square is homotopic to the identity of $\cS^n_{p-1}$.

We conclude by gluing the two squares along their common vertical side, after noticing that the composition of the two top horizontal maps is precisely $\lst_i$.
\end{proof}
Note that the map $\lst_i\colon Q^{p+1}\times\Hur_{n-p-1}(Q)\to Q^p\times\Hur_{n-p}(Q)$, as a map with codomain a product, can be described by giving two maps $Q^{p+1}\times\Hur_{n-p-1}(Q)\to Q^p$ and $Q^{p+1}\times\Hur_{n-p-1}(Q)\to \Hur_{n-p}(Q)$. The first of these maps is the projection $Q^{p+1}\times\Hur_{n-p-1}(Q)\to Q^{p+1}$ followed by the map 
$Q^{p+1}\to Q^p$ given by $(a_0,\dots,a_p)\mapsto (a_0,\dots,\hat a_i,\dots,a_p)$. The second of these maps, restricted to the slice $(a_0,\dots,a_p)\times \Hur_{n-p-1}(Q)\cong \Hur_{n-p-1}(Q)$, is the stabilisation map $\lst(a_i^{a_{i+1}\dots a_p})$.

\subsection{Action of \texorpdfstring{$\Hur(Q)$}{Hur(Q)} on \texorpdfstring{$\cS_\bullet$}{Sdot}}
\begin{nota}
\label{nota:cSdot}
 Recall Notation \ref{nota:cSndot}. We denote by $\cS_\bullet$ the disjoint union of augmented semisimplicial spaces $\coprod_{n\ge0}\cS^n_\bullet$, again leaving $Q$ implicit.
\end{nota}

For all $n,m\ge0$ there is a map of augmented semisimplicial sets $\cA_{n,\bullet}\to \cA_{n+m,\bullet}$ given by ``adding $m$ points on the right'': more precisely, we send the isotopy class of the collection of arcs $\alpha_0,\dots,\alpha_p\colon[0,1]\to D$, representing a $p$-simplex in $\cA_{n,p}$, to the isotopy class of the collection of arcs $\chi^m_n\circ\alpha_0,\dots,\chi^m_n\circ\alpha_p\colon [0,1]\to D$, where $\chi^m_n\colon D\to D$ sends $(x,y)\mapsto (\frac{m+1}{m+n+1}x,y)$. We consider the action of $\Br_n\times\Br_m$ on $\cA_{n,\bullet}$ given by projecting $\Br_n\times\Br_m\to\Br_n$ and then letting $\Br_n$ act; then the map $\cA_{n,\bullet}\to \cA_{n+m,\bullet}$ is equivariant with respect to the actions of $\Br_n\times\Br_m$ on $\cA_{n,\bullet}$ and of $\Br_{n+m}$ on $\cA_{n+m,\bullet}$, along the inclusion
of groups $\Br_n\times \Br_m\subset\Br_{n+m}$.

Similarly, the identification of sets $Q^n\times Q^m\cong Q^{n+m}$ is equivariant with respect to the action of $\Br_n\times\Br_m$ and $\Br_{n+m}$ on the two sets, respectively. Taking cartesian products, we obtain a map of augmented semisimplicial sets 
\[
\cA_n(Q)_\bullet\times Q^m\to \cA_{n+m}(Q)_\bullet,
\]
which is equivariant with respect to the action of $\Br_n\times\Br_m$ and $\Br_{n+m}$ on the two augmented semisimplicial sets. Taking homotopy quotients, we finally obtain a map of augmented semisimplicial spaces
\[
 \cS^n_\bullet\times\Hur_m(Q)\to\cS^{n+m}_\bullet.
\]
The following is immediate.
\begin{prop}
\label{prop:rightaction}
 The maps $\cS^n_\bullet\times\Hur_m(Q)\to\cS^{n+m}_\bullet$ constructed above assemble into a right action of the topological monoid $\Hur(Q)$ on the augmented semisimplicial space $\cS_\bullet$.
\end{prop}

\section{The spectral sequence argument}
Let $R$ be a commutative ring.
Recall that for an augmented semisimplicial space $X_\bullet$ there is a spectral sequence in homology, with first page $E^1_{p,q}=H_q(X_p;R)$ and limit $H_{p+q+1}(X_{-1},|X|;R)$.
For $X=\cS^n_\bullet$ we get $E^1_{p,q}=H_q(\cS^n_p;R)$, and the limit is $H_{p+q+1}(\cS^n_{-1},|\cS^n_\bullet|;R)$, which is zero for all $p,q\in\Z$ by Lemma \ref{lem:damiolini}.

For $p,q\ge0$, the $E^1$-differential $d^1_{p,q}\colon E^1_{p,q}\to E^1_{p-1,q}$ is given by the map
\[
\sum_{j=0}^p (-1)^j(d_j)_*\colon H_q(\cS^n_p;R)\to H_q(\cS^n_{p-1};R),
\]
i.e. it is the alternating sum of the maps induced in homology by the face maps (and by the augmentation).

By Proposition \ref{prop:rightaction}, we can put together these spectral sequences for varying $n\ge0$, obtaining a spectral sequence of right $H_*(\Hur(Q);R)$-modules, and in particular of right $A\otimes R$-modules. We will use this to prove Theorem \ref{thm:main1}.

\subsection{The Koszul-like complex}\label{subsec:Koszulcomplex}
We aim at proving Theorem \ref{thm:main1} by induction on $i$. The statement is obvious for $i=0$, as the ring $A\otimes R$ is finitely generated over itself. We focus henceforth on the inductive step: we fix $\nu\ge0$, assume that $H_i(\Hur(Q);R)$ is finitely generated over $A\otimes R$ for all $0\le i\le \nu$, and aim at proving that $H_{\nu+1}(\Hur(Q);R)$ is also finitely generated over $A\otimes R$.

In the previous paragraph, all $A\otimes R$-modules are meant as \emph{left} $A\otimes R$-modules, as in the statement of Theorem \ref{thm:main1}. Yet the Pontryagin ring structure of $H_*(\Hur(Q);R)$ makes $H_i(\Hur(Q);R)$ into an $A\otimes R$-bimodule, i.e. $H_i(\Hur(Q);R)$ is endowed with compatible structures of left $A\otimes R$-module and right $A\otimes R$-module.

The following definition is taken from \cite[Subsection 4.1]{EVW:homstabhur} and slightly adapted to our purposes.
\begin{defn}
\label{defn:Kcomplex}
Let $M$ be an $A\otimes R$-bimodule. We define a chain complex $K_*(M)$ of right $A\otimes R$-modules, concentrated in degrees $*\ge-1$. We set $K_p(M)=RQ^{p+1}\otimes_R M$, where for a set $S$ we denote by $RS$ the free $R$-module generated by $S$. We put on $K_p(M)$ the right $A\otimes R$-module structure coming from $M$.

The differential $d_p\colon K_p(M)\to K_{p-1}(M)$ has the form $d_p=\sum_{j=0}^i(-1)^jd_{p,j}$, where $d_{p,j}$ sends
\[
(a_0,\dots,a_p)\otimes \mu\mapsto(a_0,\dots,\hat a_j,\dots,a_p)\otimes [a_i^{a_{i+1}\dots a_p}]\cdot \mu,
\]
for $(a_0,\dots,a_p)\in Q^{p+1}$ and $\mu\in M$.
\end{defn}
Our interest in the chain complex $K_*(H_q(\Hur(Q)))$ comes from the fact that it coincides with the $q$\textsuperscript{th} row of the $E^1$-page of the spectral sequence associated with the augmented semisimplicial space $\cS_\bullet$: this is a consequence of Proposition \ref{prop:descriptiondi} and the general description of the first page of the spectral sequence associated with the skeletal filtration of an augmented semisimplicial space.

In \cite{EVW:homstabhur}, $K_*(M)$ is referred to as a ``Koszul-like complex'', and in fact in \cite{ORW:Hurwitz} it is shown that $K_*(M)$ is isomorphic to the Koszul complex of the augmented dg ring $C_*(\Hur(Q);R)$ acting on the $A$-module $M$, where the action is given by the dg ring map $C_*(\Hur(Q);R)\to A\otimes R$ sending each 0-chain to its homology class and vanishing in higher degree, and where the augmentation $C_*(\Hur(Q);R)\to R$ is the composite of the previous dg ring map with the ring map $A\otimes R\to R$
killing the positive-weight part of $A\otimes R$.
In particular $H_i(K_*(M))\cong\Tor_i^{C_*(\Hur(Q);R)}(M,R)$.
See also \cite[Remark 7.2]{ETW:shufflealgebras}: the Koszul dual of the augmented dga $C_*(\Hur(Q);R)$, i.e. $\mathrm{Ext}^*_{C_*(\Hur(Q);R)}(R;R)$, can be identified with the quantum shuffle $R$-algebra generated by the braided $R$-module $\Hom_R(RQ,R)$, i.e. the $R$-linear dual of the free $R$-module $RQ$ generated by the set $Q$. The dual $R$-coalgebra can be canonically identified, as a graded $R$-module, with $\bigoplus_{i\ge0} (RQ)^{\otimes_R i}$.

\subsection{\texorpdfstring{$G$}{G}-twists} In the following we consider each left/right/bimodule over $A\otimes R$ as a left/right/bimodule over $A$ and over $R$ by considering the canonical ring homomorphisms $A\to A\otimes R$ and $R\to A\otimes R$.
\begin{defn}\label{defn:Gtwist}
Let $M$ be an $A\otimes R$-bimodule. A \emph{$G$-twist} on $M$ is a right action of $G$ on $M$, denoted $(m,g)\mapsto m^g$ for $m\in M$ and $g\in G$, satisfying the following: 
\begin{enumerate}
\item for all $r\in R$ we have $r\cdot m=m\cdot r$;
\item for all $a,b\in Q$ we have $([a]\cdot m)^b=[a^b]\cdot m^b$ and $(m\cdot [a])^b=m^b\cdot[a^b]$;
\item for all $a\in Q$ we have $m\cdot [a]=[a]\cdot m^{a}$.
\end{enumerate}
\end{defn}

For example, for $i\ge0$ the $A\otimes R$-bimodule $H_i(\Hur(Q);R)$ admits a $G$-twist as follows. The group $G$ acts on the right on the quandle $Q$ by conjugation: the element $g$ sends $a\in Q$ to $a^g=g^{-1}ag\in Q$. We can then let $G$ act diagonally on right on the set $Q^n$; as a matter of notation, the element $g\in G$ sends $(a_1,\dots,a_n)\mapsto(a_1,\dots,a_n)^g=(a_1^g,\dots,a_n^g)$. This action commutes with the left action of $\Br_n$ on $Q^n$, and thus it induces a right action of $G$ on $\Hur_n(Q)=Q^n\qq\Br_n$. For $i\ge0$, we can then take $i$\textsuperscript{th} homology with coefficients in $R$ and consider all values of $n$ at the same time: we obtain a right action of $G$ on the weighted $A\otimes R$-bimodule $H_i(\Hur(Q);R)$.
\begin{lem}\label{lem:Gtwist}
 The right action of $G$ on the $A\otimes R$-bimodule $H_i(\Hur(Q);R)$ is a $G$-twist.
\end{lem}
\begin{proof}
The action of $R$ on $H_*(\Hur(Q);R)$ comes from the multiplication of $R$, which we assume to be commutative: this ensures that condition (1) in Definition \ref{defn:Gtwist} is satisfied.
We observe moreover that $\Hur(Q)$ is a topological monoid and that the right action of $G$ on $\Hur(Q)$ is an action by automorphisms of topological monoids: indeed for all $n,n\ge0$, the concatenation map $Q^n\times Q^m\overset{\cong}{\to} Q^{n+m}$ is $G$-equivariant; taking homotopy quotients with respect to the groups $\Br_n\times\Br_m\subset\Br_{n+m}$, we obtain that the multiplication $\Hur_n(Q)\times\Hur_m(Q)\to\Hur_{n+m}(Q)$ is $G$-equivariant. Taking homology, we obtain that $G$ acts on right on $H_*(\Hur(Q);R)$ by automorphisms of rings, and this implies that condition (2) is satisfied.

Now let $n\ge0$ and let $a\in Q$ be fixed. We denote by $\bar\fb_{n+1}\in\Br_{n+1}$ the product of standard generators $\sigma_1\dots\sigma_n$; then we have a commutative square of sets, and a commutative square of groups
 \[
  \begin{tikzcd}
     Q^n\ar[r,"-\times a"] \ar[d,"{(-)^a}"'] & Q^{n+1}\ar[d,"\bar\fb_{n+1}\cdot-"]& &    \Br_n\ar[d,equal]\ar[r,"-\times\one_1"] &\Br_{n+1}\ar[d,"(-)^{\bar\fb_{n+1}^{-1}}"]\\
   Q^n\ar[r,"a\times -"'] & Q^{n+1},& & \Br_n\ar[r,"\one_1\times-"']&\Br_{n+1},
  \end{tikzcd}
 \]
such that each map of sets is equivariant with respect to the corresponding map of groups. Taking homotopy quotients, we obtain a square of spaces, that commutes on the nose
\[
 \begin{tikzcd}
  \Hur_n(Q)\ar[r,"\rst(a)"]\ar[d,"(-)^a"'] &\Hur_{n+1}(Q)\ar[d,"\bar\fb_{n+1}\cdot-"]\\
  \Hur_n(Q)\ar[r,"\lst(a)"']&\Hur_{n+1}(Q).
 \end{tikzcd}
\]
The right vertical map in the last diagram is homotopic to the identity; we conclude that $\rst(a)\colon\Hur_n(Q)\to\Hur_{n+1}(Q)$ is homotopic to the composition $\lst(a)\circ(-)^a$; taking $i$\textsuperscript{th} homology, we obtain that also condition (3) in Definition \ref{defn:Gtwist} is satisfied.
\end{proof}

\begin{lem}\label{lem:trivialaction}
 Let $M$ be an $A\otimes R$-bimodule with a $G$-twist, and let $j\ge0$. Then $H_j(K_*(M))$ is a trivial right $A\otimes R$-module, in the sense that right multiplication by any element of $A\otimes R$ of positive weight is zero.
\end{lem}
\begin{proof}
 This is similar to \cite[Lemma 4.11]{EVW:homstabhur}, but we repeat here the argument. Since the positive-weight ideal $(A\otimes R)_+\subset A\otimes R$ is generated by the elements $[a]=[a]\otimes 1$ for $a\in Q$, it suffices to prove that $-\cdot [a]$ induces the zero map on $H_j(K_*(M))$, and for this one proves that the chain map of chain complexes of abelian groups $-\cdot [a]\colon K_*(M)\to K_*(M)$ is chain homotopic to the zero chain map.

 One defines a chain homotopy $\cH_a\colon K_*(M)\to K_{*+1}(M)$ by sending
 \[
 \cH_a\colon (a_0,\dots,a_p)\otimes \mu\mapsto (-1)^{p+1}(a_0,\dots,a_p,a)\otimes \mu^a,
 \]
 for $p\ge-1$, $(a_0,\dots,a_p)\in Q^{p+1}$ and $\mu\in M$. Using property (1) from Definition \ref{defn:Gtwist}, one checks that $(d\circ \cH_a+\cH_a\circ d)((a_0,\dots,a_p)\otimes \mu)$ is equal to $(a_0,\dots,a_p)\otimes [a]\cdot \mu^a$; using property (2), one identifies the latter with $(a_0,\dots,a_p)\otimes \mu\cdot[a]$.
\end{proof}

\subsection{Proof of Theorem \ref{thm:main1}}
Let $R$ be a Noetherian ring.
We start by observing that if $M$ is an $A\otimes R$-bimodule that is finitely generated as a right $A\otimes R$-module, then $K_*(M)$ is a chain complex of finitely generated right $A\otimes R$-modules. By Lemma \ref{lem:Anoetherian} $A\otimes R$ is Noetherian, hence also $H_i(K_*(M))$ is a finitely generated right $A\otimes R$-module; it then follows from Lemma \ref{lem:trivialaction}
that $H_j(K_*(M))$ is finitely generated also as a module over the weight-zero subring $R\subset A\otimes R$, and in particular $H_j(K_*(M))$ vanishes in high enough weight, for all $j\ge-1$.
We also observe that if $M$ is an $A\otimes R$-bimodule with a $G$-twist, then $M$ is finitely generated as a left $A\otimes R$-module if and only if it is finitely generated as a right $A\otimes R$-module.

Recall now the inductive strategy that we have established at the beginning of Subsection \ref{subsec:Koszulcomplex}.
We can apply the above discussion applies to the modules $M=H_0(\Hur(Q);R),\dots, H_\nu(\Hur(Q);R)$, and
we obtain that for $n$ large enough, with in particular $n\ge \nu+3$, the spectral sequence associated with $\cS^n_\bullet$ vanishes on a large part of its $E^2$-page: we have $E^2_{p,q}=0$ for $q\le\nu$ and $p+q\le\nu+1$. This implies in particular that the differential $d^1_{0,\nu+1}\colon E^1_{0,\nu+1}\to E^1_{-1,\nu+1}$ must be surjective, otherwise we would have $E^2_{-1,\nu+1}\neq0$ and by the vanishing of $E^2_{p,q}$ for all $p\ge1$, $q\ge0$ with $p+q=\nu+1$ we would have no other non-trivial differential hitting $E^2_{-1,\nu+1}$ and forcing its vanishing on the $E^\infty$-page, as must happen by Lemma \ref{lem:damiolini}.

The differential $d^1_{0,\nu+1}$ can be identified with the map $RQ\otimes H_{\nu+1}(\Hur(Q);R)\to H_{\nu+1}(\Hur(Q);R)$ sending $q_i\otimes m\mapsto q_i\cdot m=(q_i\otimes 1)\cdot m$; if this map is surjective for $n$ large enough, say $n\ge\bar n$, then the left $A\otimes R$-module $H_{\nu+1}(\Hur(Q);R)$ is generated by the sub-abelian group $\bigoplus_{n=0}^{\bar n} H_{\nu+1}(\Hur_n(Q);R)$. We conclude the proof of Theorem \ref{thm:main1} by observing that the latter is a finite direct sum of finitely generated abelian groups: indeed, for each $n\ge0$, the space $\Hur_n(Q)$ has the homotopy type of a finite cover of $B\Br_n$, and $B\Br_n$ is homotopy equivalent to the $n$\textsuperscript{th} unordered configuration space of points in the plane, in particular it has the the homotopy type of a finite CW complex. This concludes the proof of Theorem \ref{thm:main1}.

\section{Quasi-polynomial growth of Betti numbers}
In this section we analyse the growth of the homology groups $H_i(\Hur_n(Q);\F)$, where $\F$ is a field, for fixed $i$ and increasing $n$, and prove Corollaries \ref{cor:main1}, \ref{cor:main2}.
\subsection{Growth of weighted dimensions of \texorpdfstring{$B$}{B} and \texorpdfstring{$A$}{A}}
We first analyse the growth of the weighted dimensions (over $\Z$, aka rank) of $B$ and $A$. Recall that $\pi_0(\Hur(Q))=\pi_0(\coprod_{n\ge0}\Hur_n(Q))=\coprod_{n\ge0}\pi_0(\Hur_n(Q))$ is a weighted set; also $\pi_0^\ell(\Hur(Q))$ is a weighted set, concentrated in weights that are multiple of $\ell$. Similarly, $A$ and $B$ are weighted rings, with $B$ commutative and concentrated in weights multiple of $\ell$. Notice also that both $A$ and $B$ are free as $\Z$-modules.
\begin{prop}
 \label{prop:Agrowth}
 There is a polynomial $p_B(t)\in\Q[t]$ of degree $k(G,Q)-1$ such that, for $n$ large enough, $\dim_\Z B_{\ell n}=p_B(n)$.
 
 Similarly, there is a quasi-polynomial $p_A(t)\in\Q^\Z(t)$ of degree $k(G,Q)-1$ and period dividing $\ell$ such that $\dim_\Z A_{n}=p_A|_n(n)$ for $n$ large enough.
\end{prop}
\begin{proof}
In the proof we set $k=k(G,Q)$.
Fix a field $\F$ and let $B_\F=B\otimes \F$; then $\dim_\Z B_{\ell n}=\dim_\F B_{\F,\ell n}$, since $B$ is a free $\Z$-module. The $\F$-algebra $B_\F$ is a quotient of the weighted polynomial ring $\F[x_1,\dots,x_m]$, with variables $x_i$ put in weight $\ell$: this is witnessed by the surjective map of $\F$-algebras $\F[x_1,\dots,x_m]\twoheadrightarrow B_\F$ sending $x_i\mapsto[q_i]^\ell$. By the classical theory of the Hilbert function of a finitely generated $\F$-algebra, there is a polynomial $p_B(t)\in\Q[t]$ of degree at most $m-1$ such that, for $n$ large enough, $\dim_\F B_{\F,\ell n}=p_B(n)$.

We now want to argue that the degree of $p_B(t)$ is precisely $k-1$. For this, we will show that there exist polynomials $p^-_B(t)\in\Q[t]$ and $p^+_B(t)\in\Q[t]$ of degree $k-1$ such that, for $n$ large enough, $p^-_B(n)\le |\pi_0^\ell(\Hur(Q))_{n\ell}|\le p^+_B(n)$.
\begin{itemize}
 \item {\bf Lower bound.} Let $H\subseteq G$ be a subgroup such that $H\cap Q$ is the union of precisely $k$ conjugacy classes of $H$. Let $H'$ be the subgroup of $H$ generated by $H\cap Q$; then $Q\cap H'=Q\cap H$ also splits as a union of at least, hence precisely $k$ conjugacy classes, so we may assume to have chosen $H=H'$ at the beginning. Let $Q\cap H=\set{a_1,\dots,a_r}$ be the list of all elements in $Q\cap H$, and let $J=\set{b_1,\dots,b_k}\subset Q\cap H$ be a set representatives of the $k$ conjugacy classes in $Q\cap H\subseteq H$. For $n\ge r$ and for every splitting $n=r+n_1+\dots+n_k$, with $n_1,\dots,n_k\ge0$, we can form an element
 \[
  \hat a_1^\ell\cdot\dots\cdot\hat a_r^\ell\cdot \hat b_1^{n_1\ell}\cdot\dots\cdot\hat b_k^{n_k\ell}\in\pi_0^\ell(\Hur(Q))_{n\ell};
 \]
these elements are all distinct, for different choices of $n_1,\dots,n_k$, as witnessed by the \emph{conjugacy class partition} invariant from Subsection \ref{subsec:invcomponents}. Since there are $\binom{n-r+k-1}{k-1}$ choices for the splitting $n=r+n_1+\dots+n_k$, we can define $p_B^-(t)=\binom{t-r+k-1}{k-1}$ and obtain $|\pi_0^\ell(\Hur(Q))_{n\ell}|\ge p_B^-(n)$ for $n\ge r$.
\item {\bf Upper bound.} Let $\hat a_1^\ell\cdot\dots\cdot\hat a_n^\ell$ be an element in $\pi_0^\ell(\Hur(Q))_{n\ell}$; we want to construct a ``normal form'' for this element. First, up to reordering the factors $\hat a_i^\ell$, we may assume that there is $1\le r\le n$ such that the elements $a_1,\dots,a_r\in Q$ are all distinct, and such that for all $r+1\le j\le n$ there is $1\le i\le r$ with $a_j=a_i$.
It follows that the subgroup $H=\left<a_1,\dots,a_n\right>\subseteq G$, i.e. the \emph{image subgroup} invariant of the chosen element in $\pi_0^\ell(\Hur(Q))_{n\ell}$, can also be described as $\left<a_1,\dots,a_r\right>\subseteq G$. Let now $Q_1,\dots, Q_s$ be the conjugacy classes in $H$ in which $Q\cap H$ splits, and fix representatives $b_1\in Q_1,\dots,b_s\in Q_s$. Using repeatedly the relations from Lemma \ref{lem:relationsB} with $q_j$ chosen among $a_1,\dots,a_r$ and $q_i$ chosen among $a_{r+1},\dots,a_n$, we can achieve the situation in which each of $a_{r+1},\dots,a_n$ is one of $b_1,\dots,b_s$. Thus we have proved that each element in $\pi_0^\ell(\Hur(Q))_{n\ell}$ can be written as $\hat a_1^\ell\cdot\dots\cdot\hat a_r^\ell\cdot \hat b_1^{\ell n_1}\cdot\dots\cdot\hat b_s^{\ell n_s}$, for some $r\le m$, $a_1,\dots,a_r$ distinct elements of $Q$, and $n_1+\dots+n_s=n-r$. Making a very rough estimate, there are $2^m$ subsets in $Q$, among which one can choose $\set{a_1,\dots,a_r}$, and for each choice there are at most $n^{k-1}$ ways to choose the numbers $n_1,\dots,n_s$, since $s\le k$ and each of $n_1,\dots,n_{s-1}$ is $\le n$, and the last number $n_s$ is forced by the sum condition. Setting $p_B^+(t)=2^mt^{k-1}$, we have $|\pi_0^\ell(\Hur(Q))_{n\ell}|\le p_B^+(n)$ for all $n\ge0$.
\end{itemize}
This concludes the proof of the existence of $p_B(t)$ of degree $k-1$.
To prove that there is similarly a quasi-polynomial $p_A(t)$ such that $p_A|_n(n)=\dim_\Z(A_n)$, we use that $A$ is finitely generated as a $B$-module and invoke again the classical theory of the Hilbert function of a finitely generated graded
module over a graded algebra of finite type: this in particular ensures that the degree of $p_A(t)$ is at most $k-1$, and that the period divides $\ell$ (as $B$ is concentrated in degrees multiple of $\ell$). Finally, since $B\subseteq A$, we have that $\dim_\Z(A_{n\ell})\ge\dim_\Z(B_{n\ell})=p_B(n)$ grows in $n$ at least as a polynomial of degree $k-1$: this forces the degree of $p_A(t)$ to be $k-1$.
\end{proof}
Similarly as in Proposition \ref{prop:Agrowth}, one can show that the dimension of $(A_\one)_n$ grows like a quasi-polynomial of degree precisely $k(Q,G)-1$. This result has been proved independently by S\'eguin \cite{Seguin}, who has studied extensively the $\F$-algebra $A_\one\otimes\F$, for $\F$ a field of characteristic coprime with $|G|$.

We observe that Proposition \ref{prop:Agrowth} establishes Corollary \ref{cor:main1} for $i=0$; the general case will use Theorem \ref{thm:main1}.
\begin{proof}[Proof of Corollary \ref{cor:main1}]
 Let $i\ge0$ and fix a field $\F$; by Theorem \ref{thm:main1} the left $A\otimes\F$-module $H_i(\Hur(Q);\F)$ is finitely generated; recall also Definition \ref{defn:B} and Lemma \ref{lem:Anoetherian}, and observe that, since $A\otimes\F$ is finitely generated as a $B\otimes\F$-module, we have that $H_i(\Hur(Q);\F)$ is also finitely generated as a $B\otimes\F$-module. We now appeal again to the  
 classical theory of the Hilbert function of a finitely generated graded module over a graded algebra of finite type, together with the fact that $B\otimes \F$ is concentrated in weights multiple of $\ell$ and has weighted dimension given by a polynomial of degree $k(G,Q)$.
\end{proof}

\subsection{Proof of Corollary \ref{cor:main2}}
We fix an element $\omega\in G$ and $i\ge0$ throughout the subsection.
\begin{defn}\label{defn:Bomega}
 We denote by $B_\omega$ the quotient of $B$ by the ideal generated by the elements $[q_i]^\ell-[q_i^\omega]^\ell$.
\end{defn}
We observe that $B_\omega$ is the monoid ring of the quotient of the abelian monoid $\pi_0^\ell(\Hur(Q))$ by the relations $\hat q_i^\ell=\widehat{q_i^{\omega}}^\ell$; in particular $B_\omega$ is again a weighted ring and it is free as a $\Z$-module. Similarly, $B_\omega\otimes R$ is free as an $R$-module, for any commutative ring $R$.

\begin{lem}\label{lem:lstarstag}
 Let $g\in G$ and let $a\in Q$; then the maps $\lst(a)\colon\Hur(Q)_g\to\Hur(Q)_{ag}$ and $\rst(a^g)\colon\Hur(Q)_g\to\Hur(Q)_{ag}$ are homotopic.
\end{lem}
\begin{proof}
 The argument is similar to the one in the proof of Lemma \ref{lem:Gtwist}. We fix $n\ge0$ and prove that $\lst(a)$ and $\rst(a^g)$ are homotopic as maps $\Hur_{n}(Q)_g\to\Hur_{n+1}(Q)_{ag}$.
 Let $\tilde\fb_{n+1}\in\Br_{n+1}$ be the product of standard generators $\sigma_n\dots\sigma_1$. Denote by $Q^n_g$ the subset of $Q^n$ of sequences $(a_1,\dots,a_n)$ with total monodromy $a_1\dots a_n=g$, and define similarly $Q^{n+1}_{ag}$; then we have commutative triangles of sets and groups
 \[
  \begin{tikzcd}
     Q^n_g\ar[r,"a\times -"]\ar[dr,"-\times a^g"']  & Q^{n+1}_{ag}\ar[d,"\tilde\fb_{n+1}\cdot-"]& &   \Br_n \ar[dr,"-\times\one_1"']\ar[r,"\one_1\times-"] &\Br_{n+1}\ar[d,"(-)^{\tilde\fb_{n+1}^{-1}}"]\\
    & Q^{n+1}_{ag},& & &\Br_{n+1},
  \end{tikzcd}
 \]
such that each map of sets is equivariant with respect to the corresponding map of groups. Taking homotopy quotients, we obtain a commutative triangle of spaces
\[
 \begin{tikzcd}
  \Hur_n(Q)_g\ar[r,"\lst(a)"]\ar[dr,"\rst(a^g)"'] &\Hur_{n+1}(Q)_{ag}\ar[d,"\tilde\fb_{n+1}\cdot-"]\\
  &\Hur_{n+1}(Q)_{ag},
 \end{tikzcd}
\]
and we conclude by noticing that the right vertical map in the last diagram is homotopic to the identity.
\end{proof}

In the following proposition it is helpful to notice that if $M$ is an $A\otimes R$-bimodule with a $G$-twist (see Definition \ref{defn:Gtwist}), then the left and the right actions of $B\otimes R$ on $M$ coincide.

\begin{prop}\label{prop:HiBomegamodule}
Let $R$ be a commutative ring; then the action of $B\otimes R$ on $H_i(\Hur(Q)_\omega;R)$ factors through $B_\omega\otimes R$.
\end{prop}
\begin{proof}
It suffices to prove that for all $a\in Q$, the maps
\[
-\cdot[a]^\ell,\  -\cdot[a^\omega]^\ell\colon H_i(\Hur(Q)_\omega;R)\to H_i(\Hur(Q)_\omega;R)
\]
coincide; the previous maps can be identified with the maps induced in homology by the maps of spaces
\[
\rst(a)^\ell,\ \rst(a^\omega)^\ell\colon \Hur(Q)_\omega\to\Hur(Q)_\omega,
\]
and hence it will suffice to prove that these two maps are homotopic. As shown in the proof of Lemma \ref{lem:Gtwist}, the map $\rst(a)\colon\Hur(Q)\to\Hur(Q)$ is homotopic to the composition $\lst(a)\circ(-)^a$; we also notice that $\lst(a)\circ(-)^a=(-)^a\circ\lst(a)$, as a consequence of the fact that $a^a=a$. Taking $\ell$-fold compositions, we obtain that $\rst(a)^\ell$ is homotopic to $((-)^a)^\ell\circ\lst(a)^\ell$; we then notice that $((-a)^a)^\ell$ coincides with $(-)^{a^\ell}$, which is the identity of $\Hur(Q)$. Thus we have shown that $\rst(a)^\ell$ is homotopic to $\lst(a)^\ell$ as a map $\Hur(Q)\to\Hur(Q)$, and in particular the restricted maps $\Hur(Q)_\omega\to \Hur(Q)_\omega$ are homotopic.

By Lemma \ref{lem:lstarstag} we moreover have that the map $\lst(a)^\ell\colon\Hur(Q)_\omega\to\Hur(Q)_\omega$ is homotopic to the composition $\rst(a^{a^{\ell-1}\omega})\circ\dots\circ\rst(a^{a\omega})\circ\rst(a^\omega)$; since $a^{a^i}=a$ for all $i\ge0$, we also have $a^{a^i\omega}=a^\omega$, and eventually we obtain that $\lst(a)^\ell$ is homotopic to $\rst(a^\omega)^\ell$.
\end{proof}

We can now analyse the growth of the weighted dimension of the algebra $B_\omega$, in a similar way as we did for $B$ in Proposition \ref{prop:Agrowth}.
\begin{prop}\label{prop:Bomegagrowth}
 There is a polynomial $p_{B_\omega}(t)\in\Q[t]$ of degree $\le k(G,Q,\omega)-1$ such that, for $n$ large enough, $\dim_\Z (B_\omega)_{\ell n}=p_{B_\omega}(n)$.
\end{prop}
\begin{proof}
 The existence of a polynomial $p_{B_\omega}(t)$, possibly of degree $\ge k(G,Q,\omega)$, but satisfying $\dim_\Z (B_\omega)_{\ell n}=p_{B_\omega}(n)$ for $n$ large enough, is guaranteed by the classical theory of the Hilbert function of a finitely generated $\F$-algebra, after tensoring $B_\omega$ by some field $\F$; in fact, since $B_\omega$ is a quotient of $B$, we immediately obtain that the degree of $p_{B_\omega}(t)$ is at most the degree of $p_B(t)$, which is $k(G,Q)$.
 
 We now want to improve the upper bound on the degree of $p_{B_\omega}(t)$ to $K(G,Q,\omega)$. Recall from the proof of the upper bound in Proposition \ref{prop:Agrowth} that for $n\ge m=|Q|$, the abelian group $B_{\ell n}$ is generated by the products $[a_1]^\ell\dots[a_r]^\ell[b_1]^{\ell n_1}\dots[b_s]^{\ell n_s}$, for varying choices of:
 \begin{itemize}
  \item an integers $0\le r\le m=|Q|$;
  \item elements $a_1,\dots,a_r\in Q$;
  \item a partition $n-r=n_1+\dots+n_s$,
 \end{itemize}
 where we set $H=\left<a_1,\dots,a_r\right>\subseteq G$ (depending on the choice of $r$ and $a_1,\dots,a_r$), we set $s\le k(G,Q)$ to be the number of conjugacy classes of $Q\cap H$ in $H$, and we let $b_1,\dots,b_s$ be a system of representatives of these conjugacy classes (depending on $H$; we can choose a priori such a system of representatives for any subgroup $H\subseteq G$).

Fix now $0\le r\le m$ and elements $a_1,\dots,a_r\in Q$, and let $H$, $s$ and $b_1,\dots,b_s$ be as above. Let also $L=\left<H,\omega\right>\subseteq G$ and let $P_1,\dots,P_u$ be the conjugacy classes in which the conjugation invariant subset $Q\cap L$ of $L$ splits, for some $0\le u\le k(G,Q,\omega)$; finally, let $c_1,\dots,c_u\in Q$ be representatives. Each element $b_i$ belongs to $L\cap Q$, and inside $L$ is conjugate to some element $c_{\iota(i)}$, for a suitable function $\iota\colon\set{1,\dots,s}\to\set{1,\dots,u}$; using the relations of $B$ together with the additional relations of $B_\omega$ from Definition \ref{defn:Bomega}, we have in $B_\omega$ the equality
\[
[a_1]^\ell\dots[a_r]^\ell[b_1]^{\ell n_1}\dots[b_s]^{\ell n_s}=
[a_1]^\ell\dots[a_r]^\ell[c_{\iota(1)}]^{\ell n_1}\dots[c_{\iota(s)}]^{\ell n_s}.
\]
It follows that $(B_\omega)_{\ell n}$ is generated by all products $[a_1]^\ell\dots[a_r]^\ell[c_1]^{\ell \nu_1}\dots[c_u]^{\ell \nu_u}$, for varying $r\ge0$, $a_1,\dots,a_r\in Q$, and partitions $n-r=\nu_1+\dots+\nu_u$, where $0\le u\le k(G,Q,\omega)$ and $c_1,\dots,c_n\in Q$ depend on $r\ge0$ and $a_1,\dots,a_r\in Q$. Again by a very rough estimate, there are at most $2^mn^{k(G,Q,\omega)-1}$ such products, and this proves that $\pi_{B_\omega}(n)\le 2^mn^{k(G,Q,\omega)-1}$ for $n$ large enough, in particular the degree of $\pi_{B_\omega}(t)$ is at most $k(G,Q,\omega)-1$.
\end{proof}
Since $B_\omega$ is free as an abelian group, for any field $\F$ we have  $\dim_\F (B_\omega\otimes \F)_{\ell n}= \dim_\Z (B_\omega)_{\ell n}$.
\begin{proof}[Proof of Corollary \ref{cor:main2}]
This is now analogous to the proof of Corollary \ref{cor:main1}. For a field $\F$ and $i\ge0$, by Theorem \ref{thm:main2} we have that $H_i(\Hur(Q);\F)$ is a finitely generated $A_\one\otimes\F$-module, and by Lemma \ref{lem:Anoetherian} we have that $A_\one\otimes\F$ is finitely generated over $B\otimes \F$; it follows that $H_i(\Hur(Q);\F)$ is a finitely generated $B\otimes\F$-module, and since by Proposition \ref{prop:HiBomegamodule} the action of $B\otimes\F$ on $H_i(\Hur(Q);\F)$ factors through the projection of rings $B\otimes\F\twoheadrightarrow B_\omega\otimes\F$, we have that $H_i(\Hur(Q);\F)$ is a finitely generated $B_\omega\otimes\F$-module; the statement is now a consequence of Proposition \ref{prop:Bomegagrowth} and the classical theory of the Hilbert function of a finitely generated graded
module over a graded algebra of finite type, together with the observation that $B_\omega\otimes\F$ is concentrated in weights multiple of $\ell$.
\end{proof}

\section{Large elements as total monodromy}
In this section we prove Theorem \ref{thm:main3} and Corollary \ref{cor:main3}. We fix a commutative ring $R$ throughout the section, as well as a homological degree $i\ge0$. We assume that $Q\subset G$ is a single conjugacy class and fix a large element $\omega\in G$ as in Definition \ref{defn:omegalarge}.

\subsection{The algebra \texorpdfstring{$C$}{C}}
In this subsection we prove Theorem \ref{thm:main3} assuming that $R$ is Noetherian. We fix a Noetherian ring $R$ throughout the subsection.
\begin{defn}
 We denote by $C$ the quotient of $B$ by the relations $[q_i]^\ell=[q_j]^\ell$ for all $1\le i,j\le m$.
\end{defn}
Notice that $C$ is also a quotient of $B_\omega$; as a ring, it is isomorphic to a weighted polynomial ring in a variable of weight $\ell$; for instance, one can take the image of $[a]^\ell$ in the quotient to be such a variable, for any $a\in Q$.

Let $\bar n\ge\ell $ be such that the finitely generated $B\otimes R$-module $H_i(\Hur(Q)_\omega;R)$ is generated by the direct sum $\bigoplus_{n=0}^{\bar n-\ell}H_i(\Hur_n(Q)_\omega;R)$: the existence of such $\bar n$ is guaranteed by the assumption that $R$ is Noetherian and Theorem \ref{thm:main2}. Then we can consider the direct sum
\[
H_i(\Hur_{\ge\bar n}(Q)_\omega;R):=\bigoplus_{n\ge\bar n} H_i(\Hur_n(Q)_\omega;R)
\]
as a sub-$B\otimes R$-module of $H_i(\Hur(Q)_\omega;R)$.
\begin{lem}\label{lem:HiCmodule}
 The action of $B\otimes R$ on $H_i(\Hur_{\ge\bar n}(Q)_\omega;R)$ factors through $C\otimes R$.
\end{lem}
\begin{proof}
 Let $n\ge\bar n$ and let $x\in H_i(\Hur_n(Q)_\omega;R)$ be a homology class of the form $[a]^\ell y$, for some $y\in H_i(\Hur_{n-\ell}(Q)_\omega;R)$; then for all $b\in \Q$ we have the following:
 \begin{itemize}
  \item $[b]^\ell x=[b^\omega]^\ell x$, by Proposition \ref{prop:HiBomegamodule};
  \item $[b]^\ell x=[b]^\ell[a]^\ell y=[b^a]^\ell[a]^\ell y=[b^a]^\ell x$, using the relations of $B$.
 \end{itemize}
 It follows that if $b,c\in Q$ can be obtained from another by a sequence of conjugations by $\omega^{\pm1}$ or $a^{\pm1}$, then $[b]^\ell x=[c]^\ell x$. Since $\omega$ is large, the group $\left<\omega,a\right>$ is the entire $G$, and by assumption $Q$ is a unique conjugacy class in $G$. We conclude for any homology class $x$ of the form $[a]^\ell y$ and for any $b,c\in Q$ we have $[b]^\ell x=[c]^\ell x$.
 
 The claim now follows from the definition of $\bar n$: for $n\ge\bar n$, every class in $H_i(\Hur_n(Q)_\omega;R)$ is a linear combination of classes of the form $[a]^\ell y$, for possibly different values of $a\in Q$ and $y\in H_i(\Hur_{n-\ell}(Q)_\omega;R)$.
\end{proof}
\begin{proof}[Proof of Theorem \ref{thm:main3} for $R$ Noetherian]
By Theorem \ref{thm:main2} $H_i(\Hur(Q)_\omega;R)$ is a finitely generated $A\otimes R$-module, and by Lemma \ref{lem:Anoetherian} we have that $A\otimes R$ is finitely generated over $B\otimes R$; it follows that $H_i(\Hur(Q)_\omega;R)$ is finitely generated over $B\otimes R$; since $B\otimes R$ is Noetherian, and since $H_i(\Hur_{\ge\bar n}(Q)_\omega;R)$ is a sub-$B\otimes R$-module of $H_i(\Hur(Q)_\omega;R)$, we obtain that also $H_i(\Hur_{\ge\bar n}(Q)_\omega;R)$ is finitely generated over $B\otimes R$. It is then a consequence of Lemma \ref{lem:HiCmodule} that $H_i(\Hur_{\ge\bar n}(Q)_\omega;R)$ is in fact a finitely generated $C\otimes R$-module.

Let now $a\in Q$; the ring $C\otimes R$ is isomorphic to the polynomial algebra over $R$ generated by $[a]^\ell$, and in particular also $C\otimes R$ is Noetherian. We conclude that $H_i(\Hur_{\ge\bar n}(Q)_\omega;R)$ is not only finitely generated, but also finitely presented. Choosing $\tilde n\ge\bar n$ such that $H_i(\Hur_{\ge\bar n}(Q)_\omega;R)$ admits a presentation in weights $\le\tilde n$, we obtain that for $n>\tilde n$ the multiplication $[a]^\ell\cdot-$ induces a bijection $H_i(\Hur_n(Q)_\omega;R)\cong H_i(\Hur_{n+\ell}(Q)_\omega;R)$; this multiplication coincides with $\lst(a)^\ell_*$.

We conclude by observing that, by Lemma \ref{lem:HiCmodule}, for $n\ge\bar n$ and for varying $a\in Q$, the stabilisation maps $\lst(a)_*^\ell\colon H_i(\Hur_n(Q)_\omega;R)\to H_i(\Hur_{n+\ell}(Q)_\omega;R)$ are equal to each other.
\end{proof}
We conclude the subsection by proving Theorem \ref{thm:main3} for a general ring $R$; so in the following proof we drop the hypothesis that $R$ is Noetherian.
\begin{proof}[Proof of Theorem \ref{thm:main3} in the general case]
Let $i\ge0$, let $G,Q,\omega$ as in the statement of Theorem \ref{thm:main3}, and let $a\in Q$. Then the proof of Theorem \ref{thm:main3} for integral homology has the following direct consequence: there is $\bar n\ge0$ such that for $n\le \bar n$ the map $\lst(a)^\ell\colon\Hur_n(Q)_\omega\to\Hur_{n+\ell}(Q)_\omega$ is an integral homology isomorphism in all homological degrees $\le i$; by the universal coefficient theorem for homology, it follows that for $n\ge\bar n$ the same map induces an isomorphism in $R$-homology in the same range of degrees, for any ring $R$.

In particular, the direct sum $\bigoplus_{n\le\bar n}H_i(\Hur_n(Q)_\omega;R)$ generates $H_i(\Hur(Q)_\omega;R)$ over $B\otimes R$; we can now repeat the argument of Lemma \ref{lem:HiCmodule}, and show that the sub-$B\otimes R$-module $\bigoplus_{n\ge\bar n+\ell}H_i(\Hur_n(Q)_\omega;R)$ is in fact a $C\otimes R$-module; this implies that for $n\ge\bar n+\ell$ the stabilisation map $\lst(a)^\ell_*\colon H_i(\Hur_n(Q)_\omega;R)\to H_i(\Hur_{n+\ell}(Q)_\omega;R)$ is independent of $a\in Q$.
\end{proof}

\subsection{Homology of the group completion}
We recall the group-completion theorem by McDuff and Segal \cite{McDuffSegal}
(see also \cite[Theorem Q.4]{FM94}).
\begin{thm}[Group-completion theorem]
\label{thm:groupcompletion}
Let $R$ be a ring and let $M$ be a topological monoid; suppose that the localisation $H_*(M;R)[\pi_0(M)^{-1}]$
can be constructed by right fractions. Then the canonical
map
\[
  H_*(M;R)[\pi_0(M)^{-1}]\to H_*(\Omega BM;R)
\]
is an isomorphism of rings.
\end{thm}
Lemmas \ref{lem:Gtwist} and \ref{lem:lstarstag} imply that the multiplicative set $\pi_0(\Hur(Q))$ of the ring $H_*(\Hur(Q);R)$ satisfies the Ore condition; hence the localisation
\[
H_*(\Hur(Q);R)[\pi_0(\Hur(Q))^{-1}]
\]
can be constructed by right fractions, and Theorem \ref{thm:groupcompletion} is thus applicable to compute $H_*(\Omega B\Hur(Q))$.
\begin{proof}[Proof of Corollary \ref{cor:main3}]
For a fixed homological degree $i\ge0$, the localisation of the $A\otimes R$-module $H_i(\Hur(Q);R)$ at the multiplicative subset $\pi_0(\Hur(Q))\subset A\otimes R$ coincides with the homology $H_i(\Omega B\Hur(Q);R)$. This localisation can be constructed by right fractions; moreover the multiplicative set $\pi_0(\Hur(Q))\subset A\otimes R$ is generated multiplicatively by the finite set of elements $[a]=[a]\otimes 1\in A\otimes R$.

We define $\fw=\hat q_1^\ell\dots \hat q_m^\ell\in \pi_0^\ell(\Hur(Q))$, and denote by $[\fw^{-1}]\in B\otimes R\subset A\otimes R$ the corresponding generator; this allows us to identify the module localisation $H_i(\Hur(Q);R)[\pi_0(\Hur(Q))^{-1}]$ with $H_i(\Hur(Q);R)[\fw^-1]$. Taking the zero component $\Omega_0 B\Hur(Q)$, we can identify  $H_i(\Omega_0 B\Hur(Q);R)$ with the colimit of the following sequential diagram, where
$\alpha$ is the chosen element in $\pi_0(\Hur(Q)_{\omega})$ (but for the following diagram any other $\alpha'\in\pi_0(\Hur(Q))$ would work):
\[
\begin{tikzcd}
 H_i(\Hur_\alpha(Q);R)\ar[r,"{[\fw]\cdot-}"]& H_i(\Hur_{\fw\alpha}(Q);R)\ar[r,"{[\fw]\cdot-}"]&H_i(\Hur_{\fw^2\alpha}(Q);R)\ar[r,"{[\fw]\cdot-}"] &\dots.
\end{tikzcd}
\]
The map $[\fw]\cdot-\colon H_i(\Hur(Q)_\omega)\to H_i(\Hur(Q)_\omega)$ can be written as a composition $\lst(q_m)^\ell_*\circ\dots\circ\lst(q_1)^m_*$, and by Theorem \ref{thm:main3}, for $n$ large enough and for fixed $a\in Q$, each of the maps $\lst(q_i)^\ell_*\colon H_i(\Hur_n(Q)_\omega;R)\to H_i(\Hur_{n+\ell}(Q)_\omega;R)$ is an isomorphism and coincides with $\lst(a)^\ell_*$, where $a\in Q$ is our fixed element. It follows that the sequential diagram above can be regarded as an ``index-$m$'' subdiagram of the following diagram
\[
 \begin{tikzcd}
 H_i(\Hur_\alpha(Q);R)\ar[r,"{[a]^\ell\cdot-}"]& H_i(\Hur_{\hat a^\ell\alpha}(Q);R)\ar[r,"{[a]^\ell\cdot-}"]&H_i(\Hur_{\hat a^{2\ell}\alpha}(Q);R)\ar[r,"{[a]^\ell\cdot-}"] &\dots.
\end{tikzcd}
\]
The last diagram stabilises by Theorem \ref{thm:main3}, and its colimit is $H_i(\Omega_0 B\Hur(Q);R)$. 
\end{proof}


\bibliographystyle{amsalpha}
\bibliography{bibliography}
\end{document}